\documentclass[10pt]{article}

\DeclareMathAlphabet{\mathpzc}{OT1}{pzc}{m}{it}

\usepackage{amsmath}
\usepackage{amssymb}
\usepackage{amsfonts}
\usepackage{amsthm}
\usepackage{mathrsfs}
\usepackage{enumerate}
\usepackage[pdftex]{color, graphicx}

\newcommand{\cC}{{\mathcal C}}

\newcommand{\sL}{{\mathscr L}}

\newcommand{\pa}{\parallel}

\newcommand{\sm}{\setminus}

\newcommand{\C}{\mathcal{C}}

\newcommand{\Cp}{\mathcal{C}_P}
\newcommand{\N}{\mathcal{N}}
\newcommand{\Nh}{\mathcal{N}_H}
\newcommand{\Npp}{\mathcal{N}_{P'}}

\allowdisplaybreaks

\newtheorem{thm}{Theorem}[section]
\newtheorem{lem}[thm]{Lemma}

\newtheorem{prop}[thm]{Proposition}
\newtheorem{cor}[thm]{Corollary}

\begin{document}

\renewcommand{\thefootnote}{\arabic{footnote}}
 	
\title{{\bf Vertex positions of the generalized orthocenter and a related elliptic curve}}

\author{\renewcommand{\thefootnote}{\arabic{footnote}}
Igor Minevich and Patrick Morton}
\maketitle

\begin{section}{Introduction}

In the third of our series of papers on cevian geometry \cite{mm3}, we have studied the properties of the generalized orthocenter $H$ of a point $P$ with respect to an ordinary triangle $ABC$ in the extended Euclidean plane, using synthetic techniques from projective geometry.  This generalized orthocenter is defined as follows.  Letting $K$ denote the complement map with respect to $ABC$ and $\iota$ the isotomic map (see \cite{ac}, \cite{mm1}), the point $Q=K \circ \iota (P)$ is called the isotomcomplement of $P$.  Further, let $D, E, F$ denote the traces of $P$ on the sides of $ABC$.  The generalized orthocenter $H$ is defined to be the unique point $H$ for which the lines $HA, HB, HC$ are parallel to $QD, QE, QF$, respectively. We showed (synthetically) in \cite{mmq} that $H$ is given by the formula
$$H = K^{-1} \circ T_{P'}^{-1} \circ K(Q),$$
where $T_{P'}$ is the unique affine map taking $ABC$ to the cevian triangle $D_3E_3F_3$ of the isotomic conjugate $P'=\iota(P)$ of $P$.  The related point
$$O=K(H)=T_{P'}^{-1} \circ K(Q)$$
is the generalized circumcenter (for $P$) and is the center of the circumconic $\tilde{\mathcal{C}}_O=T_{P'}^{-1}(\mathcal{N}_{P'})$, where $\mathcal{N}_{P'}$ is the nine-point conic for the quadrangle $ABCP'$ (see \cite{mm3}, Theorems 2.2 and 2.4; and \cite{co1}, p. 84). \medskip

\begin{figure}
\[\includegraphics[width=5.5in]{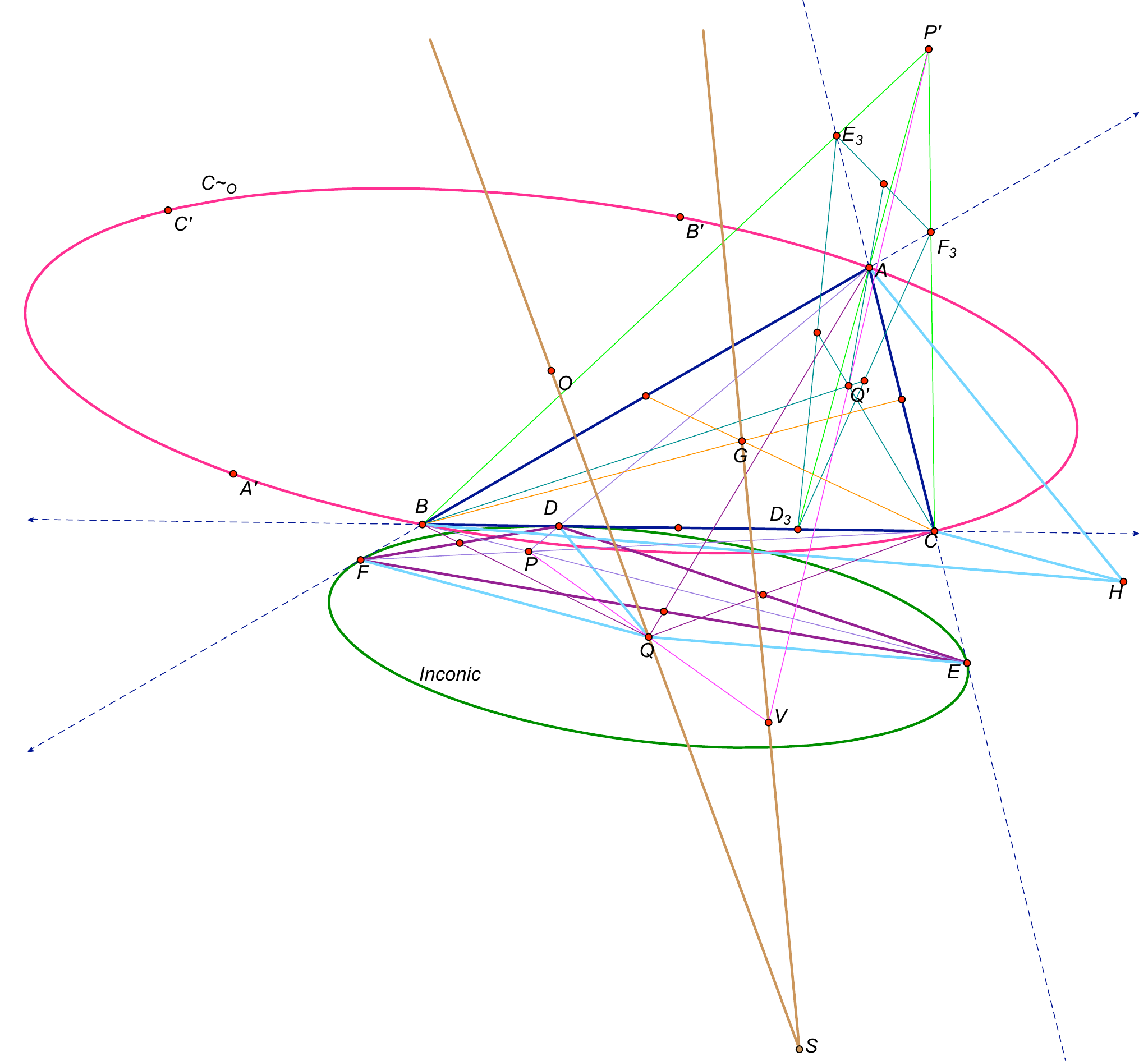}\]
\caption{The conics $\tilde{\mathcal{C}}_O$ (strawberry) and $\mathcal{I}$ (green).}
\label{fig:conics}
\end{figure}

We also showed in \cite{mm3}, Theorem 3.4, that if $T_P$ is the unique affine map taking $ABC$ to the cevian triangle $DEF$ of $P$, then the affine map $\textsf{M}=T_P \circ K^{-1} \circ T_{P'}$ is a homothety or translation which maps the circumconic $\tilde{\mathcal{C}}_O$ to the inconic $\mathcal{I}$, defined to be the conic with center $Q$ which is tangent to the sides of $ABC$ at the points $D, E, F$.  In the classical case, when $P=Ge$ is the Gergonne point of triangle $ABC$, the points $O$ and $H$ are the usual circumcenter and orthocenter, and the conics $\tilde{\mathcal{C}}_O$ and $\mathcal{I}$ are the circumcircle and incircle, respectively.  In that case the map $\textsf{M}$ taking $\tilde{\mathcal{C}}_O$ to $\mathcal{I}$ is a homothety, and its center is the insimilicenter $S$.  In general, if $G$ is the centroid of $ABC$, and $Q'=K(P)$, then the center of the map $\textsf{M}$ is the point
$$S=OQ \cdot GV = OQ \cdot O'Q', \ \ \textrm{where} \ V=PQ \cdot P'Q',$$
and $O'=T_{P}^{-1} \circ K(Q')$ is the generalized circumcenter for the point $P'$.  (See Figure \ref{fig:conics}.)  \medskip

In this paper we first determine synthetically the locus of points $P$ for which the generalized orthocenter is a vertex of $ABC$.  This turns out to be the union of three conics minus six points.  Excepting the points $A, B, C$, these three conics lie inside the Steiner circumellipse $\iota(l_\infty)$ ($l_\infty$ is the line at infinity), and each of these conics is tangent to $\iota(l_\infty)$ at two of the vertices.  (See Figure \ref{fig:locus}.)  We also consider a special case in which $H$ is a vertex of $ABC$ and the map $\textsf{M}$ is a translation, so that the circumconic $\tilde{\C}_O$ and the inconic are congruent.  (See Figures \ref{fig:2.1} and \ref{fig:2.2} in Section 2.)  In Section 3, we synthetically determine the locus of all points $P$ for which $\textsf{M}$ is a translation, which is the set of $P$ for which $S \in l_\infty$.  We determine necessary and sufficient conditions for this to occur in Theorem \ref{thm:trans}; for example, we show $\textsf{M}$ is a translation if and only if the point $P$ lies on the conic $\tilde{\C}_O$.  (This situation does not occur in the classical situation, when $P$ is the Gergonne point, since this point always lies {\it inside} the circumcircle.)  Using barycentric coordinates we show that this locus is an elliptic curve minus $6$ points.  (See Figure \ref{fig:3.3}.)  We also show that there are infinitely many points $P$ in this locus which can be defined over the quadratic field $\mathbb{Q}(\sqrt{2})$, i.e., whose barycentric coordinates can be taken to lie in this field.  In particular, given two points in this locus, a third point can be constructed using the addition on the elliptic curve.  In Section 4 we show how this elliptic curve, minus a set of $12$ torsion points, may be constructed as the locus of points $P=\textsf{A}(P_1)$, where $\textsf{A}$ runs over the affine mappings taking inscribed triangles (with a fixed centroid) on a subset $\mathscr{A}$ of a hyperbola $\cC$ (consisting of six open arcs; see equation (\ref{eqn:arcs}))  to a fixed triangle $ABC$, and where $P_1$ is a fixed point on the hyperbola (an endpoint of one of the arcs making up $\mathscr{A}$).  In another paper \cite{mms} we will show that the locus of points $P$, for which the map $\textsf{M}$ is a half-turn, is also an elliptic curve, which can be synthetically constructed in a similar way using the geometry of the triangle.  \medskip

We adhere to the notation of \cite{mm1}-\cite{mm4}: $P$ is always a point not on the extended sides of the ordinary triangles $ABC$ and $K^{-1}(ABC)$; $D_0E_0F_0=K(ABC)$ is the medial triangle of $ABC$, with $D_0$ on $BC$, $E_0$ on $CA$, $F_0$ on $AB$ (and the same for further points $D_i,E_i,F_i$); $DEF$ is the cevian triangle associated to $P$; $D_2E_2F_2$ the cevian triangle for $Q=K \circ \iota(P)=K(P')$; $D_3E_3F_3$ the cevian triangle for $P'=\iota(P)$.  As above, $T_P$ and $T_{P'}$ are the unique affine maps taking triangle $ABC$ to $DEF$ and $D_3E_3F_3$, respectively, and $\lambda=T_{P'} \circ T_P^{-1}$.  See \cite{mm1} and \cite{mm2} for the properties of these maps.  \medskip

We also refer to the papers \cite{mm1}, \cite{mm2}, \cite{mm3}, and \cite{mm4} as I, II, III, and IV respectively.  See \cite{ac}, \cite{co1}, \cite{co2} for results and definitions in triangle geometry and projective geometry.

\end{section}

\begin{section}{The special case $H=A, O=D_0$.}

We now consider the set of all points $P$ such that $H=A$ and $O=K(H)=K(A)=D_0$.  We start with a lemma.

\begin{lem}
\label{lem:EquivH=A}
Provided the generalized orthocenter $H$ of $P$ is defined, the following are equivalent:
\begin{enumerate}[(a)]
\item $H = A$.
\item $QE = AF$ and $QF = AE$.
\item $F_3$ is collinear with $Q$, $E_0$, and $K(E_3)$.
\item $E_3$ is collinear with $Q$, $F_0$, and $K(F_3)$.
\end{enumerate}
\end{lem}
\begin{proof}
(See Figure \ref{fig:locus}.)  We use the fact that $K(E_3)$ is the midpoint of segment $BE$ and $K(F_3)$ is the midpoint of segment $CF$ from I, Corollary 2.2.  Statement (a) holds iff $QE \pa AB$ and $QF \pa AC$, i.e. iff $AFQE$ is a parallelogram, which is equivalent to (b). Suppose (b) holds. Let $X = BE \cdot QF_3$. Then triangles $BXF_3$ and $EXQ$ are congruent since $QE \pa BF_3 = AB$ and $QE = AF = BF_3$. Therefore, $BX = EX$, i.e. $X$ is the midpoint $K(E_3)$ of $BE$, so $Q, F_3$, and $X = K(E_3)$ are collinear. The fact that $E_0$ is also collinear with these points follows from $K(BP'E_3) = E_0QK(E_3)$ and the collinearity of $B, P', E_3$.  Similarly, $Q, E_3, F_0$, and $K(F_3)$ are collinear. This shows (b) $\Rightarrow$ (c), (d).

Next, we show (c) and (d) are equivalent. Suppose (c) holds. The line $F_3E_0=E_0K(E_3)=K(BE_3)$ is the complement of the line $BE_3$, hence the two lines are parallel and
\begin{equation}
\frac{AF_3}{F_3B} = \frac{AE_0}{E_0E_3}.
\label{eqn:Ratios1}
\end{equation}
Conversely, if this equality holds, then the lines are parallel and $F_3$ lies on the line through $K(E_3)$ parallel to $P'E_3$, i.e. the line $K(P'E_3) = QK(E_3)$, so (c) holds. Similarly, (d) holds if and only if
\begin{equation}
\frac{AE_3}{E_3C} = \frac{AF_0}{F_0F_3}.
\label{eqn:Ratios2}
\end{equation}
A little algebra shows that (\ref{eqn:Ratios1}) holds if and only if (\ref{eqn:Ratios2}) holds.  Using signed distances, and setting $AE_0/E_0E_3 = x$, we have $AE_3/E_3C = (x + 1)/(x - 1)$.  Similarly, if $AF_0/F_0F_3 = y$, then $AF_3/F_3B = (y + 1)/(y - 1)$.  Now  (\ref{eqn:Ratios1}) is equivalent to $x = (y+1)/(y-1)$, which is equivalent to $y = (x+1)/(x-1)$, hence also to (\ref{eqn:Ratios2}).  Thus, (c) is equivalent to (d).  Note that this part of the lemma does not use that $H$ is defined.  \smallskip

Now assume (c) and (d) hold.  We will show (b) holds in this case.  By the reasoning in the previous paragraph, we have $F_3Q \pa E_3P'$ and $E_3Q \pa F_3P'$, so $F_3P'E_3Q$ is a parallelogram. Therefore, $F_3Q = P'E_3 = 2\cdot QK(E_3)$, so $F_3K(E_3) = K(E_3)Q$. This implies the triangles $F_3K(E_3)B$ and $QK(E_3)E$ are congruent (SAS), so $AF = BF_3 = QE$. Similarly, $AE = CE_3 = QF$, so (b) holds.
\end{proof}

\begin{thm}
\label{thm:locus}
The locus $\sL_A$ of points $P$ such that $H = A$ is a subset of the conic $\overline{\C}_A$ through $B, C, E_0$, and $F_0$, whose tangent at $B$ is $K^{-1}(AC)$ and whose tangent at $C$ is $K^{-1}(AB)$. Namely, $\mathscr{L}_A = \overline{\C}_A \sm \{B, C, E_0, F_0\}$.
\end{thm}
\begin{proof}
Given $E$ on $AC$ we define $F_3$ as $F_3 = E_0K(E_3) \cdot AB$, and $F$ to be the reflection of $F_3$ in $F_0$.  Then we have the following chain of projectivities ($G$ is the centroid):
\[BE \ \barwedge \ E \ \barwedge \ E_3 \ \stackrel{G}{\doublebarwedge} \ K(E_3) \ \stackrel{E_0}{\doublebarwedge} \ F_3 \ \barwedge \ F \ \barwedge \ CF.\]
Then $P = BE \cdot CF$ varies on a line or a conic. From the lemma it follows that: (a) for a point $P$ thus defined, $H = A$; and (b) if $H = A$ for some $P$, then $P$ arises in this way, i.e. $F_3$ is on $E_0K(E_3)$.  \smallskip

Now we list four cases in the above projectivity for which $H$ is undefined, namely when $P = B, C, E_0, F_0$.  Let $A_\infty, B_\infty, C_\infty$ represent the points at infinity on the respective lines $BC, AC$, and $AB$.  \smallskip

\begin{enumerate}[1.]
\item For $E = B_\infty = E_3 = K(E_3)$, we have $E_0K(E_3) = AC$ so $F_3 = A, F = B$, and $P = BE \cdot CF = B$.
\item For $E = C$, we have $E_3 = A, K(E_3) = D_0, E_0K(E_3) = D_0E_0 \pa AB, F = F_3 = C_\infty$, so $P = BE\cdot CF = C$.
\item For $E = E_0$, we have $E_3 = E_0$ and $K(E_0)$ is the midpoint of $BE_0$, so $F_3 = B, F = A$, and $P = BE\cdot CF = E_0$.
\item For $E = A$, we have $E_3 = C, K(E_3) = F_0, F_3 = F = F_0$, and $P = BE\cdot CF = F_0$.
\end{enumerate}

\begin{figure}
\[\includegraphics[width=5.5in]{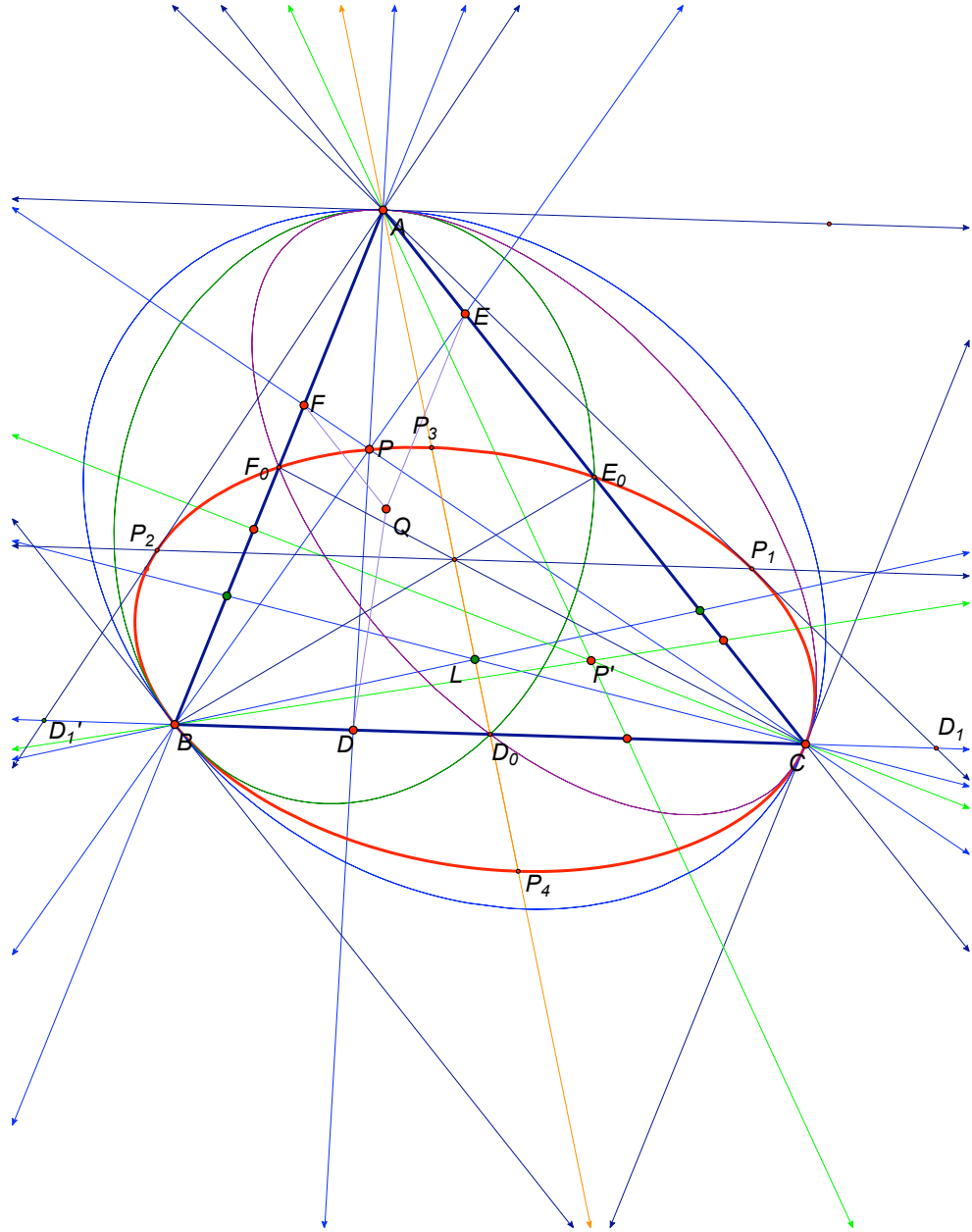}\]
\caption{The conics $\overline{\C}_A$ (red), $\overline{\C}_B$ (purple), $\overline{\C}_C$ (green), and $\iota(l_\infty)$ (blue).}
\label{fig:locus}
\end{figure}

Since the four points $B, C, E_0, F_0$ are not collinear, this shows that the locus of points $P=BE \cdot CF$ is a conic $\overline{\C}_A$ through $B, C, E_0, F_0$.  Moreover, the locus $\mathscr{L}_A$ of points $P$ such that $H = A$ is a subset of $\overline{\C}_A  \sm \{B, C, E_0, F_0\}$. \smallskip

We claim that if $E$ is any point on line $AC$ other than $A, C, E_0$, or $B_\infty$, then $P$ is a point for which $H$ is well-defined. First, $E_3$ is an ordinary point because $E \ne B_\infty$. Second, because $E \ne B_\infty$, the line $E_0K(E_3)$ is not a sideline of $ABC$. The line $E_0K(E_3)$ intersects $AB$ in $A$ if and only if $K(E_3)$ lies on $AC$, which is true only if $E_3 = B_\infty$. The line $E_0K(E_3)$ intersects $AB$ in $B$ iff $K(E_3)$ is on $BE_0$, which holds iff $E_3$ is on $K^{-1}(B)B = BE_0$, and this is the case exactly when $E = E_3 = E_0$.  Since $K(E_3)$ lies on $K(AC)=D_0F_0$, the line $E_0K(E_3)$ is parallel to $AB$ iff $K(E_3)=D_0$, giving $E_3=A$ and $E=C$. Thus, the line $E_0K(E_3)$ intersects $AB$ in an ordinary point which is not a vertex, so $F_3$ and $F$ are not vertices and $P=BE\cdot CF$ is a point not on the sides of $ABC$. \smallskip

It remains to show that $P$ does not lie on the sides of the anticomplementary triangle of $ABC$. If $P$ is on $K^{-1}(AB)$ then $F=F_3 = C_\infty$, which only happens in the excluded case $E=C$ (see Case 2 above). If $P$ is on $K^{-1}(AC)$ then $E= B_\infty$, which is also excluded. If $P$ is on $K^{-1}(BC)$ then $P'$ is also on $K^{-1}(BC)$ so $Q=K(P')$ is on $BC$.  \smallskip

To handle the last case, suppose $Q$ is on the same side of $D_0$ as $C$. Then $P'$ is on the opposite side of line $AD_0$ from $C$, so it is clear that $CP'$ intersects $AB$ in the point $F_3$ between $A$ and $B$.  If $Q$ is between $D_0$ and $C$, then $F_3$ is between $A$ and $F_0$ (since $F_0, C,$ and $G$ are collinear), and it is clear that $F_3E_0$ can only intersect $BC$ in a point outside of the segment $D_0C$, on the opposite side of $C$ from $Q$.  But this is a contradiction, since by construction $F_3, E_0$, and $K(E_3)$ are collinear, and $Q=K(P')$ lies on $K(BE_3)=E_0K(E_3)$.  On the other hand, if the betweenness relation $D_0 * C * Q$ holds, then $F_3$ is between $B$ and $F_0$, and it is clear that $F_3E_0$ can only intersect $BC$ on the opposite side of $B$ from $C$.  This also applies when $P'=Q$ is a point on the line at infinity, since then $F_3=B$, and $B, E_0$ and $Q=A_\infty$ (the point at infinity on $BC$) are not collinear, contradicting part (c) of Lemma \ref{lem:EquivH=A}.   A symmetric argument applies if $Q$ is on the same side of $D_0$ as $B$, using the fact that parts (c) and (d) of Lemma \ref{lem:EquivH=A} are equivalent. Thus, no point $P$ in $\overline{\C}_A  \sm \{B, C, E_0, F_0\}$ lies on a side of $ABC$ or its anticomplementary triangle, and the point $H$ is well-defined; further, $H=A$ for all of these points.  \smallskip

Finally, by the above argument, there is only one point $P$ on $\overline{\C}_A$ that is on the line $K^{-1}(AB)$, namely $C$, and there is only one point $P$ on $\overline{\C}_A$ that is on the line $K^{-1}(AC)$, namely $B$, so these two lines are tangents to $\overline{\C}_A$.
\end{proof}

This theorem shows that the locus of points $P$, for which the generalized orthocenter $H$ is a vertex of $ABC$, is the union of the conics $\overline{\C}_A \cup \overline{\C}_B \cup \overline{\C}_C$ minus the vertices and midpoints of the sides.  The Steiner circumellipse is tangent to the sides of the anticomplementary triangle $K^{-1}(ABC)$,  so the conic $\overline{\C}_A$, for instance, has the double points $B, C$ in common with $\iota(l_\infty)$.  Since the conic $\overline{\C}_A$ lies on the midpoints $E_0$ and $F_0$, which lie inside $\iota(l_\infty)$, it follows from Bezout's theorem that the set $\overline{\C}_A - \{B,C\}$ lies entirely in the interior of $\iota(l_\infty)$, with similar statements for $\overline{\C}_B$ and $\overline{\C}_C$.  \medskip

In the next proposition and its corollary, we consider the special case in which $H=A$ and $D_3$ is the midpoint of $AP'$.  We will show that, in this case, the map $\textsf{M}$ is a translation.  (See Figure \ref{fig:2.2}.)  We first show that this situation occurs.  \medskip

\begin{lem}
\label{lem:equilateral}
If the equilateral triangle $ABC$ has sides of length $2$, then there is a point $P$ with $AP \cdot BC=D$ and $d(D_0,D)=\sqrt{2}$, such that $D_3$ is the midpoint of the segment $AP'$ and $H=A$.
\end{lem}

\begin{figure}
\[\includegraphics[width=5in]{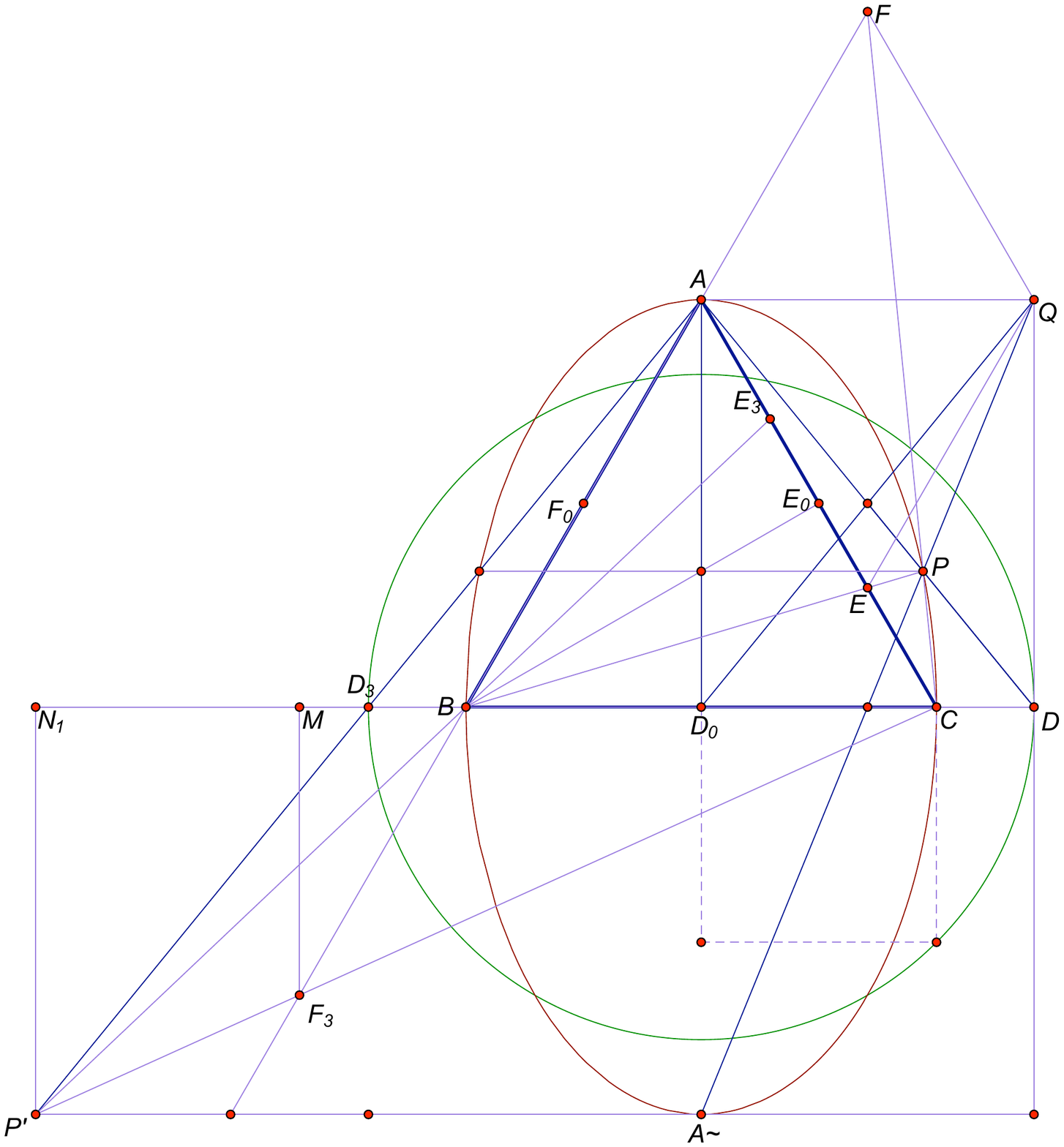}\]
\caption{Proof of Lemma \ref{lem:equilateral}}
\label{fig:2.1}
\end{figure}

\begin{proof}
(See Figure \ref{fig:2.1}.) We will construct $P'$ such that $D_3$ is the midpoint of $AP'$ and $H=A$, and then show that $P$ satisfies the hypothesis of the lemma.  The midpoint $D_0$ of $BC$ satisfies $D_0B = D_0C = 1$ and $AD_0 = \sqrt3$. Let the triangle be positioned as in Figure \ref{fig:2.1}.  Let $\tilde A$ be the reflection of $A$ in $D_0$, and let $D$ be a point on $BC$ to the right of $C$ such that $D_0D = \sqrt{2}$. In order to ensure that the reflection $D_3$ of $D$ in $D_0$ is the midpoint of $AP'$, take $P'$ on $l=K^{-2}(BC)$ with $P'\tilde A = 2\sqrt2$ and $P'$ to the left of $\tilde A$. Then $Q = K(P')$ is on $K^{-1}(BC)$, to the right of $A$, and $AQ = \sqrt{2}$. Let $E_3$ and $F_3$ be the traces of $P'$ on $AC$ and $BC$, respectively.  \smallskip

We claim $BF_3 = \sqrt{2}$. Let $M$ be the intersection of $BC$ and the line through $F_3$ parallel to $AD_0$. Then triangles $BMF_3$ and $BD_0A$ are similar, so $F_3M = \sqrt{3} \cdot MB$. Let $N_1$ be the intersection of $BC$ and the line through $P'$ parallel to $AD_0$. Triangles $P'N_1C$ and $F_3MC$ are similar, so
\[\frac{F_3M}{MC} = \frac{P'N_1}{N_1C} = \frac{AD_0}{P'\tilde A + 1} = \frac{\sqrt{3}}{2\sqrt{2} + 1}.\]
Therefore,
\[\frac{\sqrt{3}}{2\sqrt{2} + 1} = \frac{F_3M}{MC} = \frac{\sqrt{3} \cdot MB}{MB + 2}\]
which yields that $MB = 1/\sqrt{2}$. Then $BF_3=\sqrt{2}$ is clear from similar triangles. \smallskip

Now, let $F$ be the reflection of $F_3$ in $F_0$ (the midpoint of $AB$). Then $AQF$ is an equilateral triangle because $m(\angle FAQ)=60^\circ$ and $AQ \cong BF_3 \cong AF$, so $\angle AQF \cong \angle AFQ$. Therefore, $QF \pa AC$. It follows that the line through $F_0$ parallel to $QF$ is parallel to $AC$, hence is a midline of triangle $ABC$ and goes through $D_0$.  Hence, the point $O$, which is the intersection of the lines through $D_0, E_0, F_0$, parallel to $QD, QE, QF$, respectively, must be $D_0$, giving that $H=K^{-1}(O)=A$.  Clearly, $P=AD \cdot CF$ is a point outside the triangle $ABC$, not lying on an extended side of $ABC$ or its anticomplementary triangle, which satisfies the conditions of the lemma.
\end{proof}

The next proposition deals with the general case, and shows that the point $P$ we constructed in the lemma  lies on a line through the centroid $G$ parallel to $BC$.  In this proposition and in the rest of the paper, we will use various facts about the center $Z$ of the cevian conic $\cC_P=ABCPQ$, which we studied in detail in the papers \cite{mm2} and \cite{mm3}.  Recall that $Z$ lies on the nine-point conic $\Nh$.  We also recall the definition of the affine reflection $\eta$ from II, p. 27, which fixes the line $GV$, with $V = PQ \cdot P'Q'$, and moves points parallel to the line $PP'$.

\begin{prop}
\label{prop:HA}
Assume that $H=A, O=D_0$, and $D_3$ is the midpoint of $AP'$.  Then the circumconic $\tilde{\C}_O = \iota(l)$, where $l=K^{-1}(AQ)=K^{-2}(BC)$ is the line through the reflection $\tilde{A}$ of $A$ in $O$ parallel to the side $BC$.  The points $O, O', P, P'$ are collinear, with $d(O,P')=3d(O,P)$, and the map $\textsf{M}$ taking $\tilde{\C}_O$ to the inconic $\mathcal{I}$ is a translation.  In this situation, the point $P$ is one of the two points in the intersection $l_G \cap \tilde{\C}_O$, where $l_G$ is the line through the centroid $G$ which is parallel to $BC$.
\end{prop}

\begin{proof}
(See Figure \ref{fig:2.2}.)  Since the midpoint $R_1'$ of segment $AP'$ is $D_3$, lying on $BC$, $P'$ lies on the line $l$ which is the reflection of $K^{-1}(BC)$ (lying on $A$) in the line $BC$.  It is easy to see that this line is $l=K^{-2}(BC)$, and hence $Q=K(P')$ lies on $K^{-1}(BC)$.  From I, Corollary 2.6 we know that the points $D_0, R_1'=D_3$, and $K(Q)$ are collinear.  Since $K(Q)$ is the center of the conic $\Npp$ (the nine-point conic of quadrilateral $ABCP'$; see III, Theorem 2.4), which lies on $D_0$ and $D_3$, $K(Q)$ is the midpoint of segment $D_0D_3$ on $BC$.  Applying the map $T_{P'}^{-1}$ gives that $O=T_{P'}^{-1}(K(Q))$ is the midpoint of $T_{P'}^{-1}(D_3D_0)=AT_{P'}^{-1}(D_0)$.  It follows that $T_{P'}^{-1}(D_0)=\tilde{A}$ is the reflection of $A$ in $O$, so that $\tilde{A} \in \tilde{\C}_O$.  Moreover, $K(A)=O$, so $\tilde{A}=K^{-1}(A)$ lies on $l=K^{-1}(AQ) \pa BC$.   \smallskip

Next we show that $\tilde{\C}_O = \iota(l)$, where the image $\iota(l)$ of $l$ under the isotomic map is a circumconic of $ABC$ (see Lemma 3.4 in \cite{mm4}).  It is easy to see that $\iota(\tilde{A})= \tilde{A}$, since $\tilde{A} \in AG$ and $AB\tilde{A}C$ is a parallelogram.  Therefore, both conics $\tilde{\C}_O$ and $\iota(l)$ lie on the $4$ points $A,B,C, \tilde{A}$.  To show they are the same conic, we show they are both tangent to the line $l$ at the point $\tilde{A}$.  From III, Corollary 3.5 the tangent to $\tilde{\C}_O$ at $\tilde{A}=T_{P'}^{-1}(D_0)$ is parallel to $BC$, and must therefore be the line $l$.  To show that $l$ is tangent to $\iota(l)$, let $L$ be a point on $l \cap \iota(l)$.  Then $\iota(L) \in l \cap \iota(l)$.  If $\iota(L) \neq L$, this would give three distinct points, $L, \iota(L)$, and $\tilde{A}$, lying on the intersection $l \cap \iota(l)$, which is impossible.   Hence, $\iota(L)=L$, giving that $L$ lies on $AG$ and therefore $L=\tilde{A}$.  Hence, $\tilde{A}$ is the only point on $l \cap \iota(l)$, and $l$ is the tangent line.  This shows that $\tilde{\C}_O$ and $\iota(l)$ share $4$ points and the tangent line at $\tilde{A}$, proving that they are indeed the same conic.  \smallskip

From this we conclude that $P=\iota(P')$ lies on $\tilde{\C}_O$.  Hence, $P$ is the fourth point of intersection of the conics $\tilde{\C}_O$ and $\Cp=ABCPQ$.  From III, Theorem 3.14 we deduce that $P= \tilde{Z}=\textsf{R}_OK^{-1}(Z)$, where $\textsf{R}_O$ is the half-turn about $O$; and we showed in the proof of that theorem that $\tilde{Z}$ is a point on the line $OP'$.  Hence, $P, O, P'$ are collinear, and applying the affine reflection $\eta$ gives that $O'=\eta(O)$ lies on the line $PP'$, as well (see III, Theorem 2.4).  Now, $Z$ is the midpoint of $HP=AP$, since $\textsf{H}=K \circ \textsf{R}_O$ is a homothety with center $H=A$ and similarity factor $1/2$.  Since $Z$ lies on $GV$, where $V=PQ \cdot P'Q'$ (II, Prop. 2.3), it is clear that $P$ and $Q$ are on the opposite side of the line $GV$ from $P', Q'$, and $A$.  The relation $K(\tilde{A})=A$ means that $\tilde{A}$ and also $O$ are on the opposite side of $GV$ from $A$ and $O'$.  Also, $J=K^{-1}(Z)=\textsf{R}_O(\tilde{Z})=\textsf{R}_O(P)$ lies on the line $GV$ and on the conic $\tilde{\C}_O$.  This implies that $O$ lies between $J$ and $P$, and applying $\eta$ shows that $O'$ lies between $J$ and $P'$.  Hence, $OO'$ is a subsegment of $PP'$, whose midpoint is exactly $J=K^{-1}(Z)$, since this is the point on $GV$ collinear with $O$ and $O'$.  Now the map $\eta$ preserves distances along lines parallel to $PP'$ (see II, p. 27), so $JO' \cong JO \cong OP \cong O'P'$, implying that $OO'$ is half the length of $PP'$.  Furthermore, segment $QQ'=K(PP')$ is parallel to $PP'$ and half as long.  Hence, $OO' \cong QQ'$, which implies that $OQQ'O'$ is a parallelogram.  Consequently, $OQ \pa O'Q'$, and III, Theorems 3.4 and 3.9 show that $\textsf{M}$ is a translation.  Thus, the circumconic $\tilde{\C}_O$ and the inconic $\mathcal{I}$ are congruent in this situation.  This argument implies the distance relation $d(O,P') =3d(O,P)$.  \smallskip

The relation $O'Q' \pa OQ$ implies, finally, that $T_P(O'Q') \pa T_P(OQ)$, or $K(Q')P \pa A_0Q = AQ$, since $O'=T_P^{-1}K(Q')$ from \cite{mmq}, Theorem 6; $T_P(Q')=P$ from I, Theorem 3.7; $T_P(O)=T_P(D_0)=A_0$; and $A_0$ is collinear with $A$ and the fixed point $Q$ of $T_P$ by I, Theorem 2.4.  Hence, $PG=PQ'=PK(Q')$ is parallel to $AQ$ and $BC$. \smallskip
\end{proof}

\begin{figure}
\[\includegraphics[width=5.5in]{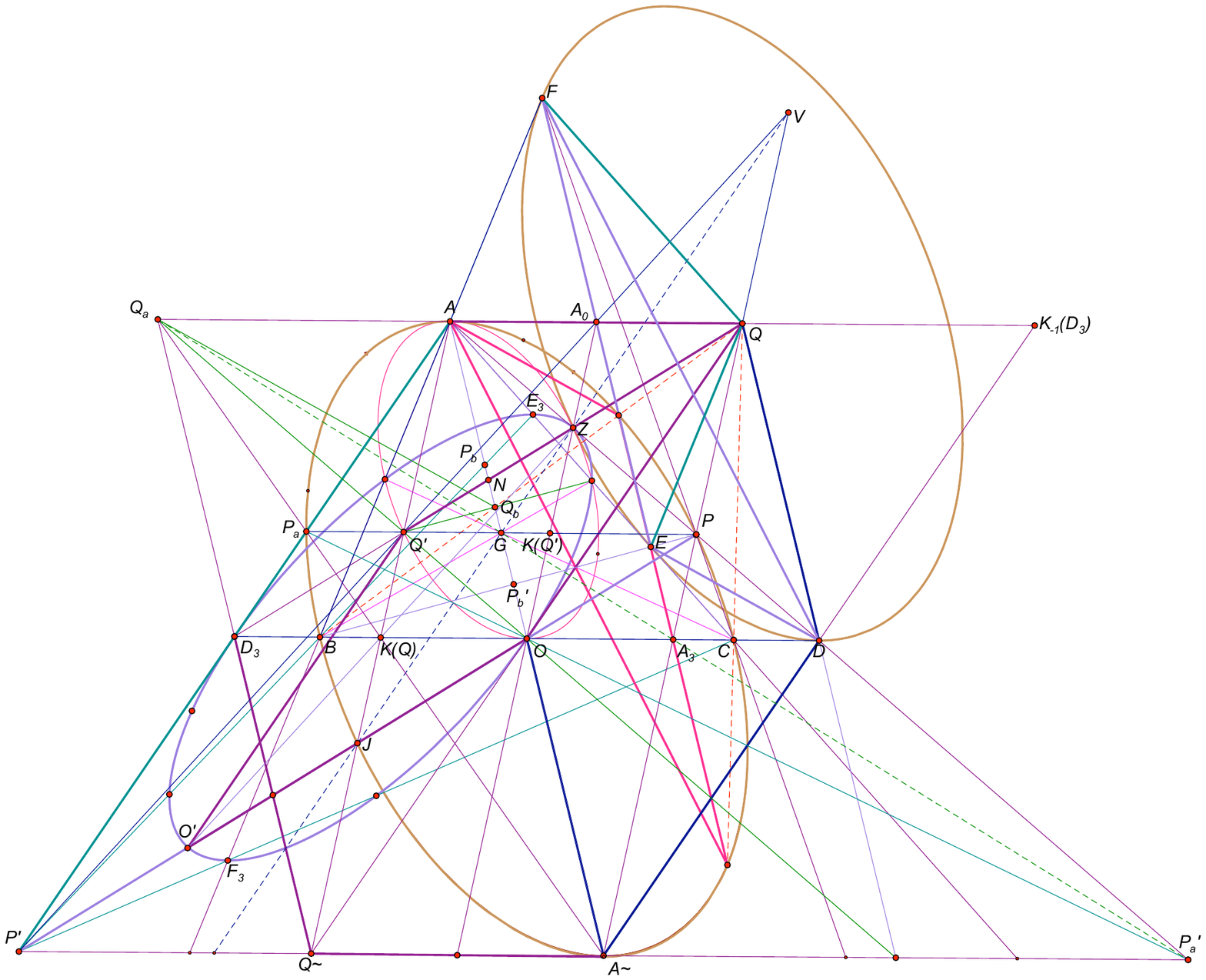}\]
\caption{The case $H=A, O=D_0$, and midpoint of $AP'=D_3$.}
\label{fig:2.2}
\end{figure}

There are many interesting relationships in the diagram of Figure \ref{fig:2.2}.  We point out several of these relationships in the following corollary.

\begin{cor}
\label{cor:HArel} Assume the hypotheses of Proposition \ref{prop:HA}.
\begin{enumerate}[a)]
\item If $Q_a$ is the vertex of the anticevian triangle of $Q$ (with respect to $ABC$) opposite the point $A$, then the corresponding point $P_a$ is the second point of intersection of the line $PG$ with $\tilde{\C}_O$.
\item The point $A_3=T_P(D_3)$ is the midpoint of segment $OD$ and $P$ is the centroid of triangle $ODQ$.
\item The ratio $\frac{OD}{OC}=\sqrt{2}$.
\end{enumerate}
\end{cor}

\begin{proof}
The anticevian triangle of $Q$ with respect to $ABC$ is $T_{P'}^{-1}(ABC)=Q_aQ_bQ_c$.  (See I, Cor. 3.11 and III, Section 2.)  Since $D_3$ is the midpoint of $AP'$, this gives that $T_{P'}^{-1}(D_3)=A$ is the midpoint of $T_{P'}^{-1}(AP')=Q_aQ$.  Therefore, $Q_a$ lies on the line $AQ=K^{-1}(BC)$, so $P_a'=K^{-1}(Q_a)$ lies on the line $l$ and is the reflection of $P'$ in the point $\tilde{A}$.  Thus, the picture for the point $P_a$ is obtained from the picture for $P$ by performing an affine reflection about the line $AG=A\tilde{A}$ in the direction of the line $BC$.  This shows that $P_a$ also lies on the line $PG \pa BC$.  The conic $\tilde{\C}_O$ only depends on $O$, so this reflection takes $\tilde{\C}_O$ to itself.  This proves a). \smallskip

To prove b) we first show that $P$ lies on the line $Q\tilde{A}$.  Note that the segment $K(P'\tilde{A})=AQ$ is half the length of $P'\tilde{A}$, so $P'\tilde{A} \cong Q_aQ$.  Hence, $Q_aQ\tilde{A}P'$ is a parallelogram, so $Q\tilde{A} \cong Q_aP'$.  Suppose that $Q\tilde{A}$ intersects line $PP'$ in a point $X$.  From the fact that $K(Q)$ is the midpoint of $D_3D_0$ we know that $Q$ is the midpoint of $K^{-1}(D_3)A$.  Also, $D_3Q'$ lies on the point $\lambda(A)=\lambda(H)=Q$, by II, Theorem 3.4(b) and III, Theorem 2.7.  It follows that $K^{-1}(D_3), P=K^{-1}(Q'), P'=K^{-1}(Q)$ are collinear and $K^{-1}(D_3)QX \sim P'\tilde{A}X$, with similarity ratio $1/2$, since $K^{-1}(D_3)Q$ has half the length of $P'\tilde{A}$.  Hence $d(X, K^{-1}(D_3)) = \frac{1}{2} d(X, P')$.  On the other hand, $d(O,P) = \frac{1}{3} d(O,P')$, whence it follows, since $O$ is halfway between $P'$ and $K^{-1}(D_3)$ on line $BC$, that $d(P, K^{-1}(D_3)) =\frac{1}{2} d(P, P')$.  Therefore, $X=P$ and $P$ lies on $Q\tilde{A}$.  \smallskip

Now, $\textsf{P}=AD_3OQ$ is a parallelogram, since $K(AP')=OQ$, so opposite sides in $AD_3OQ$ are parallel.  Hence, $T_P(\textsf{P})=DA_3A_0Q$ is a parallelogram, whose side $A_3A_0=T_P(D_3D_0)$ lies on the line $EF=T_P(BC)$.  Applying the dilatation $\textsf{H}=K\textsf{R}_O$ (with center $H=A$) to the collinear points $Q, P, \tilde{A}$ shows that $\textsf{H}(Q), Z$, and $O$ are collinear.  On the other hand, $O=D_0, Z$, and $A_0$ are collinear by \cite{mmq}, Corollary 5 (since $Z=R$ is the midpoint of $AP$), and $A_0$ lies on $AQ$ by I, Theorem 2.4.  This implies that $A_0=\textsf{H}(Q)=AQ \cdot OZ$ is the midpoint of segment $AQ$, and therefore $A_3$ is the midpoint of segment $OD$.  Since $P$ lies on the line $PG$, $2/3$ of the way from the vertex $Q$ of $ODQ$ to the opposite side $OD$, and lies on the median $QA_3$, it must be the centroid of $ODQ$.  This proves b). \smallskip

To prove c), we apply an affine map taking $ABC$ to an equilateral triangle.  It is clear that such a map preserves all the relationships in Figure \ref{fig:2.2}.  Thus we may assume $ABC$ is an equilateral triangle whose sidelengths are $2$. By Lemma \ref{lem:equilateral} there is a point $P$ for which $AP \cdot BC=D$ with $D_0D=\sqrt{2}, O=D_0$, and $D_3$ the midpoint of $AP'$.  Now Proposition \ref{prop:HA} implies the result, since the equilateral diagram has to map back to one of the two possible diagrams (Figure \ref{fig:2.2}) for the original triangle.
\end{proof}

By Proposition \ref{prop:HA} and III, Theorem 2.5 we know that the conic $\overline{\C}_A$ lies on the points $P_1, P_2, P_3, P_4$, where $P_1$ and $P_2=(P_1)_a$ are the points in the intersection $\tilde{\C}_O \cap l_G$ described in Corollary \ref{cor:HArel}, and $P_3=(P_1)_b, P_4=(P_1)_c$ are the anti-isotomcomplements of the points $(Q_1)_b, (Q_1)_c$, since these points all have the same generalized orthocenter $H=A$.  (See Figure \ref{fig:locus}.)  It can be shown that the equation of the conic $\overline{\C}_A$ in terms of the barycentric coordinates of the point $P=(x,y,z)$ is $xy+xz+yz=x^2$ (see \cite{mmq}).  Furthermore, the center of $\bar \C_A$ lies on the median $AG$, $6/7$-ths of the way from $A$ to $D_0$.  \medskip

\noindent {\bf Remarks.}
\noindent 1. The polar of $A$ with respect to the conic $\bar \C_A$ is the line $l_G$ through $G$ parallel to $BC$.  This holds because the quadrangle $BCE_0F_0$ is inscribed in $\bar \C_A$, so its diagonal triangle, whose vertices are $A, G$, and $BC \cdot \l_\infty$, is self-polar. Thus, the polar of $A$ is the line $l_G$. \smallskip

\noindent 2. The two points $P$ in the intersection $\bar \C_A \cap l_G$ have tangents which go through $A$.  This follows from the first remark, since these points lie on the polar $a=l_G$ of $A$ with respect to $\bar \C_A$.  As a result, the points $D$ on $BC$, for which there is a point $P$ on $AD$ satisfying $H = A$, have the property that the ratio of unsigned lengths $DD_0/D_0C \le \sqrt 2$.  This follows from the fact that $\bar \C_A$ is an ellipse: since it is an ellipse for the equilateral triangle, it must be an ellipse for any triangle.  Then the maximal ratio $DD_0/D_0C$ occurs at the tangents to $\bar \C_A$ from $A$; and we showed above that for these two points $P$, $D = AP \cdot BC$ satisfies $DD_0/D_0C = \sqrt 2$. \bigskip

\end{section}

\begin{section}{The locus of points $P$ for which $\textsf{M}$ is a translation.}

We can characterize the points $P$, for which $\textsf{M}$ is a translation, as follows.  We will have occasion to use the fact that $\textsf{M}=T_P \circ K^{-1} \circ T_{P'}$ is symmetric in the points $P$ and $P'$, since $T_P \circ K^{-1} \circ T_{P'}=T_{P'} \circ K^{-1} \circ T_P$.  This follows easily from the fact that the maps $T_P \circ K^{-1}$ and $T_{P'} \circ K^{-1}$ commute with each other.  See III, Proposition 3.12 and IV, Lemma 5.2.

\begin{thm}
\label{thm:trans}
Let $P$ and $P'$ be ordinary points not on the sides or medians of $ABC$ or $K^{-1}(ABC)$.  Then the map $\textsf{M}=T_P \circ K^{-1} \circ T_{P'}$ is a translation if and only if any one of the following statements holds.
\begin{enumerate}[1.]
\item $OQQ'O'$ is a parallelogram;
\item $P$ is on the circumconic $\tilde{\cC}_O$;
\item $O$ and $O'$ lie on $PP'$;
\item $Z$ lies on $QQ'$;
\item The signed ratio $\frac{GZ}{ZV} = \frac{1}{3}$;
\item $U = K^{-1}(Z)=K(V)$.
\end{enumerate}
\end{thm}

\begin{proof}
It is easy to see that $\textsf{M}$ is a translation if and only if $OQQ'O'$ is a parallelogram, since $\textsf{M}(O)=Q$ and $\textsf{M}(O')=Q'$ and the center of $\textsf{M}$ is the point $S=OQ \cdot GV = OQ \cdot O'Q'$.  (See III, Theorem 3.4 and the proof of III, Theorem 3.9.)  Thus, we will prove that the statements (2)-(6) are equivalent to (1).  \smallskip

First we prove that (1) $\Rightarrow$ (4).  If $OQQ'O'$ is a parallelogram, then $OQ \pa O'Q'$.  By IV, Propsition 3.10, $q=OQ$ is the tangent to the conic $\cC_P=ABCPQ$ at $Q$ and $q'=O'Q'$ is the tangent to $\cC_P$ at $Q'$.  It follows that $q \cdot q'$ is on $l_\infty$, which is the polar of the center $Z$ of $\cC_P$.  Therefore, $QQ'$ lies on $Z$. \smallskip

Conversely, assume (4).  Then $Z \in QQ'$ implies that $q \cdot q'$ lies on $l_\infty$, so $OQ \pa O'Q'$, giving that $S \in l_\infty$ and $\textsf{M}$ is a translation.  Hence, (4) $\Rightarrow$ (1).  Furthermore, (1) $\Rightarrow$ (3), as follows.  $\textsf{M}$ is a translation so $Q\textsf{M}(Q) \cong O\textsf{M}(O) = OQ$, i.e. $Q$ is the midpoint of $O\textsf{M}(Q)$, where $\textsf{M}(Q)=T_{P'} \circ K^{-1} \circ T_P(Q)=T_{P'}(P')$.  But $Q$ is also the midpoint of $PV$, so triangles $PQO$ and $VQ\textsf{M}(Q)$ are congruent, giving $\textsf{M}(Q)V \pa OP$.  We know $\textsf{M}(Q)V = K^{-1}(PP')$ by II, Proposition 2.3(f) and IV, Theorem 3.11(7.).  Since the line through $P$ parallel to $K^{-1}(PP')$ is $PP'$,  $O$ lies on $PP'$.  Hence, $\eta(O)=O'$ also lies on $PP'$, giving (3).  \smallskip

Now (4) holds if and only if $Z$ is the midpoint of $QQ'$ ($Z$ lies on $GV$, the fixed line of $\eta$, and $\eta(Q)=Q'$).  The point $V$ is the midpoint of segment $K^{-1}(PP')$ (II, Proposition 2.3), so $K(V)$ is the midpoint of segment $PP'$ and $K^2(V)$ is the midpoint of $K(PP')=QQ'$.  Hence, (4) holds if and only if $K^2(V)=Z$, which holds if and only if $K^{-1}(Z) = K(V)$.  Thus, (4) $\iff$ (6).  This allows us to show (4) $\Rightarrow$ (2), as follows.  Since (4) also implies (1) and (3), we have that $OO' = QQ' = \frac{1}{2}PP'$.  Also, $K(V)$ on $GV$ is the midpoint of $PP'$ and $OO'$.  Since $OQ \pa GV$ and $QP'$ intersects $GV$ at $G$, it is clear that $O$ and $P$ lie on the same side of line $GV$.  Hence, $O$ must be the midpoint of $PK(V)$ (the dilation with center $U=K(V)$ takes $OO'$ to $PP'$). Now $K(V) = K^{-1}(Z) \in \tilde \cC_O$, since $Z \in \N_H = K(\tilde \cC_O$); so $P = \textsf{R}_O(K^{-1}(Z))$ lies on $\tilde \cC_O$, hence (2).  Thus (4) $\Rightarrow$ (2).  \smallskip

We next show that (2) $\Rightarrow$ (3).  Assume that $P$ lies on $\tilde{\C}_O$.  Then $P$ is the fourth point of intersection of the conics $\tilde{\C}_O$ and $\Cp=ABCPQ$.  From III, Theorem 3.14 we deduce that $P= \tilde{Z}=\textsf{R}_O \circ K^{-1}(Z)$.  Furthermore, $\tilde{Z}$ lies on $OP'$.  Hence, $P, O, P'$ are collinear, and applying the affine reflection $\eta$ gives that $O'$ lies on the line $PP'$, as well.  \smallskip

Now suppose that (3) holds, so that $O$ lies on $PP'$.  From II, Corollary 2.2 we know that $T_P(P')$ lies on $PP'$, so that IV, Theorem 3.11 implies that $O=OQ \cdot PP'=T_P(P')$.  Let $\tilde H = T_P^{-1}(H)=T_{P'}^{-1}(Q)$, as in III, Theorem 2.10.  Note that
$$\textsf{M}(\tilde H) = T_P \circ K^{-1} \circ T_{P'}(\tilde H) = T_P \circ K^{-1}(Q)=T_P(P')=O.$$
Hence, $\textsf{M}(\tilde HO)=OQ$.  Part III, Lemma 3.8 says that $O$ is the midpoint of $\tilde HQ$, so $\textsf{M}(\tilde HO) \cong \tilde HO$.  (Note that $\tilde H = T_{P'}^{-1}(Q) \neq T_{P'}^{-1}\circ K(Q)=O$.)  The result of III, Theorem 3.4 says that $\textsf{M}$ is a homothety or translation.  A homothety expands or contracts all segments on lines by the same factor $k$, so if $\textsf{M}$ were not a translation, the factor $k=\pm 1$.   But $k \neq 1$ since $\textsf{M}$ is not the identity map and $k \neq -1$ since it preserves the orientation of the segment $\tilde HO$ on the line $OQ$.  Hence, $\textsf{M}$ must be a translation. This proves (3) $\Rightarrow$ (1), and therefore $(1) \Rightarrow (4) \Rightarrow (2) \Rightarrow (3) \Rightarrow (1)$.  \smallskip

Furthermore, (5) $\iff$ (6).  If $U=K^{-1}(Z)=K(V)$, then taking signed distances yields $GZ=GK(U)=\frac{1}{2}UG=\frac{1}{4}GV$, which gives that $\frac{GZ}{ZV}=\frac{1}{3}$.  Conversely, if $\frac{GZ}{ZV}=\frac{1}{3}$, then $\frac{GZ}{GV}=\frac{1}{4}$, which implies $Z=K^2(V)$, since $K^2(V)=K(K(V))$ is the unique point $X$ on $GV$ for which the signed ratio $\frac{GX}{GV}=\frac{1}{4}$.  Thus, (5) $\iff$ (6) $\iff$ (4) (from above).  \smallskip

This completes the proof that (1)-(6) are equivalent.
\end{proof}

\begin{figure}
\[\includegraphics[width=5.5in]{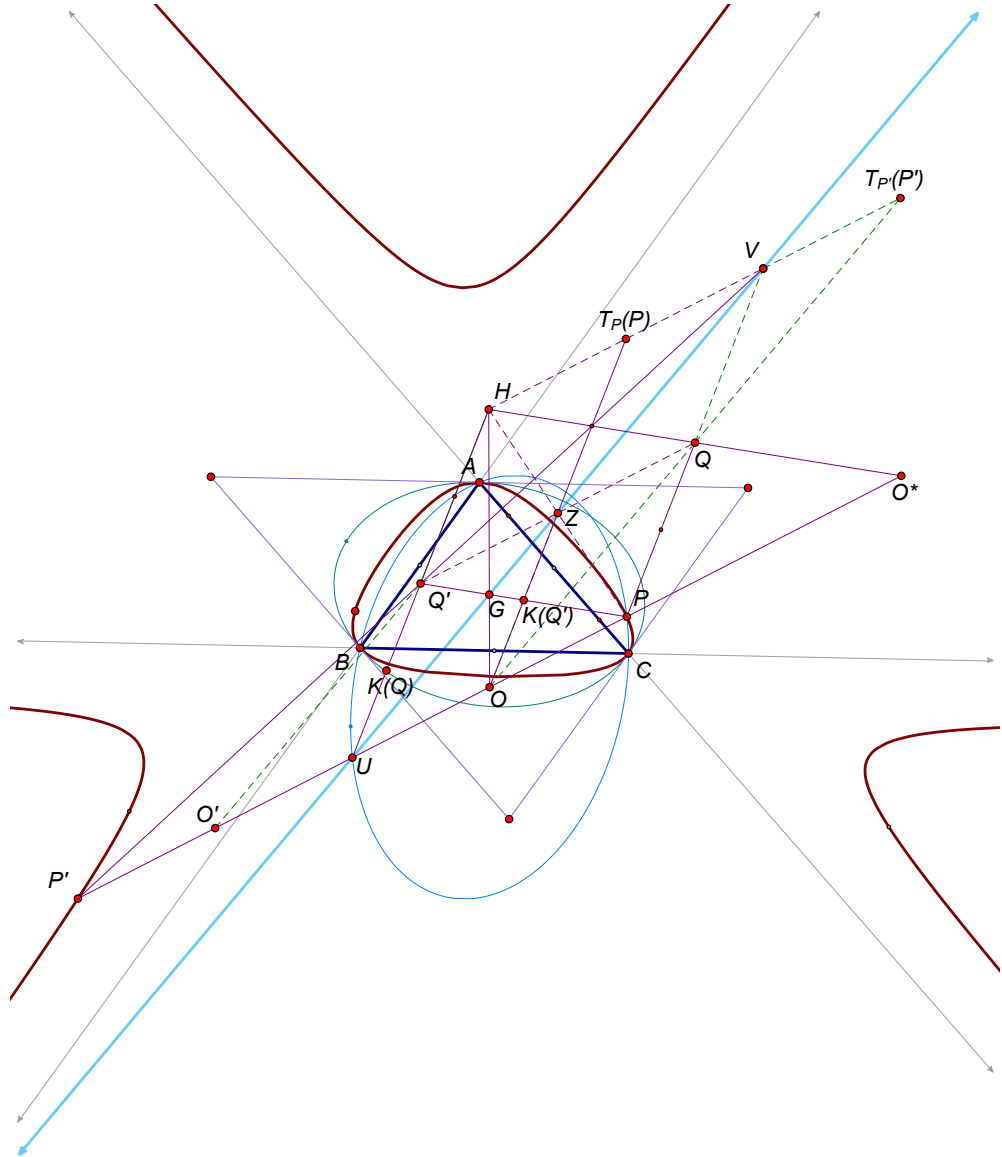}\]
\caption{Locus of $\textsf{M}$ a translation (curve $\mathcal{E}_S$ in brown) with Steiner circumellipse (teal) and $\tilde \cC_O$ (blue).}
\label{fig:3.3}
\end{figure}

\begin{cor} Under the hypotheses of Theorem \ref{thm:trans}, if $\textsf{M}$ is a translation:
\begin{enumerate}[1.]
\item $HK^{-1}(Z)PV$ is a parallelogram;
\item $T_P(P)$ is the midpoint of segment $HV$;
\item $T_P(P') = O$;
\item The points $P',O',U=K^{-1}(Z), O, P$ are equally spaced on line $PP'$;
\item $OH$ is tangent to the conic $\cC_P=ABCPQ$ at $H$.
\end{enumerate}
\label{cor:M}
\end{cor}

\begin{proof}
(See Figure \ref{fig:3.3}.)  Statements (3) and (4) were proved in the course of proving (3) $\Rightarrow$ (1) and (4) $\Rightarrow$ (2) above.  With $U=K^{-1}(Z)=K(V)$, (4) gives that that $UO =\frac{1}{2}UP$, so $UO \cong K(UP) = ZQ'$.  This shows that $UO \pa ZQ'$ and $OZQ'U$ is a parallelogram.  Thus, $K^{-1}(OZQ'U)=HUPV$ is also a parallelogram.  In particular, $QQ' \pa PP' \pa HV$.  Since $\textsf{M}(U) = \textsf{M} \circ K^{-1}(Z) = Z$ (by the Generalized Feuerbach Theorem in III) and $Z$ is the midpoint of segment $UV$, the translation $\textsf{M}$ maps parallelogram $OZQ'U$ to $QV \textsf{M}(Q')Z$, where $\textsf{M}(Q') = T_P \circ K^{-1} \circ T_{P'}(Q') = T_P(P)$.  As $Z$ is the center of parallelogram $HUPV$ and $O$ is the midpoint of the side $UP$, while $Q$ is the midpoint of side $PV$, it follows that $Q'$ is the midpoint of side $HU$ and $\textsf{M}(Q') = T_P(P)= OZ \cdot HV$ is the midpoint of $HV$, proving statement (2).   Finally, let $O^* = PP' \cdot QH$. Triangles $PO^*Q$ and $UO^*H$ are similar ($UH \pa PV=PQ$) and $PQ = \frac{1}{2}UH$, so $PQ$ is the midline of triangle $UO^*H$ and $PO^* = UP = \frac{1}{2}PP'$. This implies $\frac{PO^*}{O^*P'} = -\frac{1}{3} = -\frac{PO}{OP'}$, and therefore $O^*$ is the harmonic conjugate of $O$ with respect to $P$ and $P'$. Thus, $O^*$ is conjugate to $O$ with respect to the polarity induced by $\cC_P$.  As in the proof of the theorem, $q=OQ$ is tangent to $C_P$ at $Q$, so $Q$ is also conjugate to $O$.  Thus, the polar of $O$ is $o=QO^* = QH$ and since $H$ and $Q$ are on $\cC_P$ (III, Theorem 2.8), this implies that $OH$ is the tangent to $\cC_P$ at $H$. This proves (5).
\end{proof}

\noindent {\bf Remark.} The statements in Corollary \ref{cor:M} are actually all equivalent to the map $\textsf{M}$ being a translation.  For the sake of brevity, we leave this verification to the interested reader. \medskip

The set of points $P=(x,y,z)$, for which $\textsf{M}$ is a translation, can be determined using barycentric coordinates.  This is just the set of points for which $S \in l_\infty$.  It can be shown (see \cite{mm0}, eq. (8.1)) that homogeneous barycentric coordinates of the point $S$ are
$$S=(x(y+z)^2,y(x+z)^2,z(x+y)^2),$$
where $P=(x,y,z)$.  Thus, the the locus in question has the projective equation
$$\mathcal{E_S}: \ \ x(y+z)^2+y(x+z)^2+z(x+y)^2=0.$$
Setting $z=1-x-y$, so that $(x,y,z)$ are absolute barycentric coordinates, the set of ordinary points for which $\textsf{M}$ is a translation has the affine equation
\begin{equation}
(3x+1)y^2+(3x+1)(x-1)y+x^2-x=0.
\label{eqn:nform}
\end{equation}
Since the discriminant of this equation, with respect to $y$, is $D=(x-1)(3x+1)(3x^2-6x-1)$, this curve is birationally equivalent to
$$Y^2=(X-1)(3X+1)(3X^2-6X-1).$$
Using $X=\frac{u-1}{u+3}, Y=\frac{8v}{(u+3)^2}$, this equation can be written in the form
$$\mathcal{E}_S': \ v^2=u(u^2+6u-3),$$
which is an elliptic curve with $j$-invariant $j=54000=2^{4}3^{3}5^{3}$ and infinitely many points defined over real quadratic fields.  Thus, we see that there are infinitely many points for which $\textsf{M}$ is a translation.  Note that $\mathcal{E}_S'$ is isomorphic to the curve (36A2) in Cremona's tables \cite{cre} (via the substitution $u=x-2,v=y$).  Consequently, $\mathcal{E}_S'$ has the torsion points $T=\{\tilde O, (0,0), (1, \pm 2), (-3, \pm 6)\}$ ($\tilde O$ is the base point) and rank $r=0$ over $\mathbb{Q}$.  These $6$ points correspond to the vertices of triangle $ABC$ and the infinite points on its sides; the latter points are $(0,1,-1), (1,0,-1), (1,-1,0)$.  \medskip

It is not hard to calculate, using the equation (\ref{eqn:nform}) and the equation $xy+xz+yz=x^2$ for $\overline{\cC}_A$ that the intersection $\mathcal{E}_S \cap \overline{\cC}_A$ consists of the points $B=(0,1,0)$ and $C=(0,0,1)$, with intersection multiplicity $2$ at both points, together with the points
\begin{equation}
P=\left(\frac{1}{3},\frac{1+\sqrt{2}}{3},\frac{1-\sqrt{2}}{3}\right) \ \ \textrm{and} \ \ P_a=\left(\frac{1}{3},\frac{1-\sqrt{2}}{3},\frac{1+\sqrt{2}}{3}\right),
\label{eqn:P}
\end{equation}
where $P$ and $P_a$ are the points pictured in Figure \ref{fig:2.2}.  (These points are labeled $P_1$ and $P_2$ in Figure \ref{fig:locus}.)  That these are the correct points follows from the fact that
$$P-G=P-\left(\frac{1}{3},\frac{1}{3},\frac{1}{3}\right) = \frac{\sqrt{2}}{3}(0,1,-1),$$
and therefore $PG \pa BC$.  The affine coordinates of $P$ on (\ref{eqn:nform}) are $(x,y)=\left(\frac{1}{3},\frac{1+\sqrt{2}}{3}\right)$, which corresponds to the point $\tilde P=(u,v)=(3,6\sqrt{2})$ on $\mathcal{E}_S'$.  The double of the latter point is $[2]\tilde P=\left(\frac{1}{2},\frac{\sqrt{2}}{4}\right)$, and $[4]\tilde P=\left(\frac{169}{8},-\frac{2483 \sqrt{2}}{32}\right)$.  Using \cite{si}, Theorem VII.3.4 (p. 193) with $p=2$ over the local field $K=\mathbb{Q}_2(\sqrt{2})$, the coordinates of the last point show that $\tilde P$ is a point of infinite order on $\mathcal{E}_S'$, and therefore $P$ is a point of infinite order on (\ref{eqn:nform}).  Hence, there are infinitely many points on $\mathcal{E}_S$ which have coordinates in the field $\mathbb{Q}(\sqrt{2})$.  \medskip

It follows from this calculation that the only points, other than the vertices of $ABC$, in the intersection $\mathcal{E}_S \cap \mathscr{L}$ of $\mathcal{E}_S$ and the locus $\mathscr{L}=\mathscr{L}_A \cup \mathscr{L}_B \cup \mathscr{L}_C$ of Section 2 are the $6$ points obtained by permuting the coordinates of the point $P$ in (\ref{eqn:P}).  There are, however, $6$ more important points on the curve $\mathcal{E}_S$.  These are the intersections of $\mathcal{E}_S$ with the medians of triangle $ABC$, which are found by setting two variables equal to each other in the equation for $\mathcal{E}_S$.  This yields the following six points on $\mathcal{E}_S$, paired with their isotomic conjugates:
\begin{align*}
P_1 &= (1, -2+\sqrt{3}, -2+\sqrt{3}), \ \ P_1'=(1, -2-\sqrt{3},-2-\sqrt{3}) \\
P_2 &= (-2+\sqrt{3},1,-2+\sqrt{3}), \ \ P_2'=(-2-\sqrt{3},1,-2-\sqrt{3}) \\
P_3 &= (-2+\sqrt{3},-2+\sqrt{3},1), \ \ P_3'=(-2-\sqrt{3},-2-\sqrt{3},1).
\end{align*}
These correspond to the six points
$$(x,y)=\left(1\pm \frac{2}{3}\sqrt{3},\mp \frac{\sqrt{3}}{3} \right), \ \left(\pm \frac{\sqrt{3}}{3}, 1\mp \frac{2}{3}\sqrt{3}\right) , \ \left(\pm \frac{\sqrt{3}}{3}, \pm \frac{\sqrt{3}}{3}\right)$$
on (\ref{eqn:nform}); and to the six points
$$(u,v)=(-3 \pm 2\sqrt{3},0), \ (3 \pm 2\sqrt{3},12 \pm 6\sqrt{3}), \ (3 \pm 2\sqrt{3},-12 \mp 6\sqrt{3})$$
on $\mathcal{E}_S'$.  Together with the points in $T$, these points form a torsion group $T_{12}$ of order $12$ defined over $\mathbb{Q}(\sqrt{3})$, with $T_{12} \cong \mathbb{Z}_2 \oplus \mathbb{Z}_2 \oplus \mathbb{Z}_3$.  The points in $T_{12}$ are the points which are excluded in Theorem \ref{thm:trans} and Corollary \ref{cor:M}. In particular, there are only two excluded points on each median, for which $\textsf{M}$ is a translation. \smallskip

The equation (\ref{eqn:nform}) is a special case of the equation
$$E_a: \ \ (ax+1)y^2+(ax+1)(x-1)y+x^2-x=0,$$
which we call the {\it geometric normal form} of an elliptic curve.  It can be shown that for real values of $a \notin \{3, 0, -1, 9\}$, the set of points (not a vertex or an infinite point on the sides of $ABC$) on this elliptic curve is the locus of points $P$ for which the map $\textsf{M}$ is a homothety with ratio $k=\frac{4}{a+1}$; and every elliptic curve defined over $\mathbb{R}$ is isomorphic to a curve in this form.

\end{section}

\begin{section}{Constructing the elliptic curve.}

In this section we will use the results of the previous section to give a geometric construction of the elliptic curve $\mathcal{E}_S$.  We start with the following lemma.

\begin{lem}
\label{lem:GZ}
Assume that $P$ is a point for which the map $\textsf{M}$ is a translation.  Then the line $GZ = GV$ does not intersect the conic $\cC_P$, which is a hyperbola.
\end{lem}

\begin{proof}
We will use the characterization of $\cC_P$ as the set of points $Y$ for which $P, Y$, and $T_P(Y)$ are collinear (II, Corollary 2.2).  \smallskip

Let $Y$ be a point on $GV$ and $Y'=PT_P(Y) \cdot GV$ the projection of $Y_P=T_P(Y)$ onto $GV$ from $P$.  The mapping $Y \rightarrow Y_P$ is projective, since $T_P$ is an affine map, so the mapping $\pi: \ Y  \barwedge Y'$ is a projectivity from $GV$ to itself.  We will show that $\pi$ has no invariant points.  This will imply the lemma, since if $Y \in \cC_P$, then $Y$ lies on $PT_P(Y)$, implying that $Y=Y'$. \smallskip

We will show that the projectivity $\pi$ has order $3$ by showing that $\pi$ coincides with the projectivity $UZV \barwedge ZVU$ on $GV$.  First, $\pi(U) = Z$, because $T_P(U) = T_P(K^{-1}(Z))=Z$ is already on $GV$.  Also, since $Z$ is the midpoint of $QQ'$, $T_P(Z)$ is the midpoint of $T_P(QQ')=QP$.  This implies that $\pi(Z) = QP \cdot GV=V$.  Now, $T_P(V)$ is the intersection of $T_P(PQ)=QT_P(P)$ and $T_P(P'Q')=OP=PP'$ by Corollary \ref{cor:M}.  Hence, $\pi(V)=PT_P(V) \cdot GV =PP' \cdot GV=U$. \smallskip

Since $\pi$ has order $3$, it cannot have any invariant points.  See \cite{co1}, p. 43 or \cite{co2}, p.35, Exercises.  Finally, since $GV$ lies on the center $Z$ of $\cC_P$, but does not intersect $\cC_P$, the conic must be a hyperbola.  This completes the proof.
\end{proof}

Thus, the line $GV$ is an exterior line of $\cC_P$ (\cite{co1}, p. 72), so its pole $V_\infty$ is an interior point, which implies that the line $GV_\infty \pa PP'$ is a secant for the conic $\cC_P$, and therefore meets $\cC_P$ in two points $E$ and $F$.  (These are different points from the similarly named points in Figure \ref{fig:2.2}.) Hence, as $\eta$ fixes the line $EF$ and maps the conic $\cC_P$ to itself (II, p. 27), we have $\eta(E)=F$ and $G$ on $GV$ is the midpoint of segment $EF$.  But $EF=GV_\infty$ is the polar of $V$ with respect to $\cC_P$, so $VE$ and $VF$ are tangent to $\cC_P$ at $E$ and $F$, respectively.  We choose notation so that $E$ is the intersection of $GV_\infty$ with the branch of the hyperbola through $P'$ and $Q'$, which exists since $P'$ and $Q'$ are on the same side of the line $GV$.

\begin{figure}
\[\includegraphics[width=5.5in]{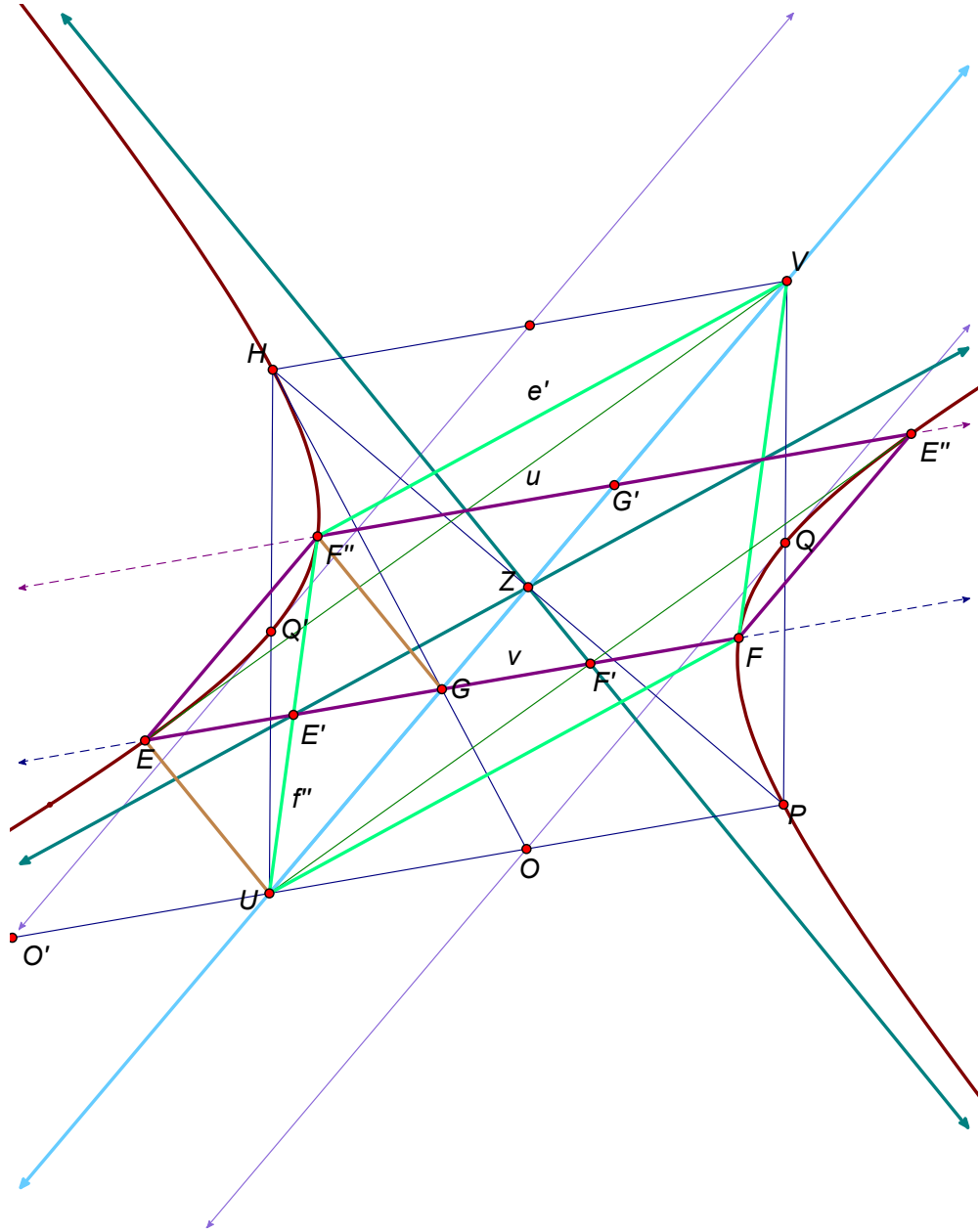}\]
\caption{The parallelogram $HVPU$ and conic $\cC_P$.}
\label{fig:4.1}
\end{figure}

\begin{prop}
Assume $P$ is a point for which $\textsf{M}$ is a translation.  If $E'$ and $F'$ are the midpoints of segments $EG$ and $GF$, then the lines $ZE'$ and $ZF'$ are the asymptotes of $\cC_P$.
\label{prop:asymp}
\end{prop}

\begin{proof}
(See Figure \ref{fig:4.1}.)  We know $Z=K(U)$ is the midpoint of segment $UV$ and the center of $\cC_P$.  If we rotate the tangents $EV$ and $FV$ by a half-turn about $Z$, we obtain two lines $E''U$ and $F''U$ through $U$ which are also tangent at points $E''$ and $F''$ respectively.  In particular, $Z$ is the midpoint of segments $EE''$, $FF''$, and $UV$. This implies $F''V \pa UF$ and $Z$ is the center of the parallelogram $VF''UF$.  Now {\it define} $E'$ to be the midpoint of $UF''$.  The line $ZE'$ is halfway between $F''V$ and $UF$, hence $ZE' \pa F''V$. \smallskip

Next we show that $EF''G'G$ is a parallelogram, where $G'=\textsf{R}_Z(G)$.  Now, $Z$ is the midpoint of $EE''$ and $FF''$, so $EF''E''F$ is a parallelogram.  Since $G$ is the midpoint of $EF$, $G'$ is the midpoint of $E''F''$. This implies $F''G' \cong EG$, which proves $EF''G'G$ is a parallelogram.  Also, using part (6.) of Theorem \ref{thm:trans} it is easy to see that $G$ is the midpoint of $UG'$ so $UG \cong GG' \cong EF''$ and $EF''GU$ is a parallelogram, with center $E'$.  This verifies that $E'$ is the midpoint of segment $EG$.  \smallskip

But $E'$ on $EG = v$ implies $V$ lies on its polar $e'$.  Also, $E'$ is on $UF'' = f''$, so $F''$ lies on $e'$.  Together, this implies $e'=F''V$, so from the first paragraph of the proof, $e' \pa ZE'$.  Hence, $e', ZE'$, and $l_\infty$ are concurrent.  The dual of this statement says that $E', l_\infty \cdot e'$, and $Z$ are collinear.  Thus, the infinite point $l_\infty \cdot e'$ lies on $ZE'$, which its own polar!   Hence, $l_\infty \cdot e'$ must lie on the conic and $ZE'$ must be an asymptote.  Applying the map $\eta$ shows that $ZF'=\eta(ZE')$ is also an asymptote.
\end{proof}

We now consider a fixed configuration of points, as in Figure \ref{fig:4.1}, consisting of the parallelogram $HUPV$, its center $Z$, the point $O$ which is the midpoint of side $UP$, the point $G=UV \cdot HO$, the midpoints $Q,Q'$ of opposite sides $HU$ and $PV$, and the points $O', P'$ which are the affine reflections of the points $O, P$ through the line $UV=GZ$ in the direction of the line $UP$, together with the conic $\cC=PQHQ'P'$.  By Theorem \ref{thm:trans} and Corollary \ref{cor:M} this configuration arises from a triangle $ABC$ and the point $P$ (not on the sides of $ABC$ or $K^{-1}(ABC)$), for which the map $\textsf{M}$ is a translation, and such a configuration certainly exists because it can be taken to be the image under an affine map of the configuration constructed in Lemma \ref{lem:equilateral} and Proposition \ref{prop:HA}.  For this configuration the conclusions of Lemma \ref{lem:GZ} and Proposition \ref{prop:asymp} hold, so that $\cC$ is a hyperbola.  Our focus now is on finding all triangles $A_1B_1C_1$ inscribed in the conic $\cC=\cC_P$ for which the map $\textsf{M}_P$ corresponding to $A_1B_1C_1$ is a translation.  This will lead us to a synthetic construction of the elliptic curve $\mathscr{E}_S$ discussed in Section 3.  \medskip

Let $A_1$ be any point on the conic $\cC=PQHQ'P'$, and define $D_0=K(A_1)$, where $K$ is the dilation about $G$ with signed ratio $-1/2$.  Further, let $\cC(A_1)$ be the reflection of the conic $\C$ in the point $D_0$.  If the conics $\cC(A_1)$ and $\C$ intersect in two points $B_1, C_1$, then $A_1B_1C_1$ is the unique triangle with vertex $A_1$ and centroid $G$ which is inscribed in $\C$.  This is because $D_0$ must be the midpoint of side $B_1C_1$ in any such triangle, and lying on $\C$, $B_1$ and $C_1$ must both lie on $\cC(A_1)$.  Since $\cC(A_1)$ is the reflection of $\C$ in $D_0$, its asymptotes  $c = \textsf{R}_{D_0}(a)$ and $d = \textsf{R}_{D_0}(b)$ are parallel to the respective asymptotes $a$ and $b$ of $\C$.  It follows that $\C \cap \cC(A_1)$ can consist of at most two points other than the infinite points on the asymptotes.

\begin{lem} The conics $\cC(A_1)$ and $\C=PQHQ'P'$ intersect in two ordinary points if and only if $A_1$ does not lie on either of the closed arcs of $\C$ between the lines $EF=GV_\infty$ and $P'P$.
\label{lem:arcs}
\end{lem}

\begin{proof} It suffices to prove the lemma for the configuration pictured in Figures \ref{fig:4.2} and \ref{fig:4.3}, since any two configurations for which $\textsf{M}$ is a translation are related by an affine map.  In particular, Corollary \ref{cor:M} shows that one configuration can be mapped to any other by an affine map taking the parallelogram $HUPV$ for the one configuration to the corresponding parallelogram for the other.  \smallskip

Let $c = \textsf{R}_{D_0}(a)$ and $d = \textsf{R}_{D_0}(b)$ be the asymptotes of $\cC(A_1)$, where $a=ZF'$ and $b=ZE'$ are the asymptotes of $\C$.
When $A_1=E$, then $D_0=K(E)=F'$ lying on $a=ZF'$, so the lines $a,c$ coincide.  Then $\C$ and $\cC(A_1)$ have the common tangent $a=c$, so they intersect with multiplicity at least $2$ at $a \cdot l_\infty$.  They also intersect with multiplicity $1$ at $b \cdot l_\infty$, since they have different tangents at that point ($D_0=F'$ is on $a$ but not $b$, so $b \neq d$).  Hence, they can have at most one ordinary point in common.  However, reflecting in $D_0$ (lying on $a=c$ and therefore in the exterior of $\C$), any ordinary intersection of $\C$ and $\cC(A_1)$ would yield a second intersection, so the two conics can't have any ordinary points in common.  On the other hand, if $A_1=P'$ on arc $EP'$, then $D_0=K(P')=Q$ lies on $\C \cap \cC(A_1)$ and also on the tangent $OQ$ to $\C$, so that $\C$ and $\cC(A_1)$ touch at $D_0=Q$.  Thus, they don't intersect in any other ordinary point. \smallskip

\begin{figure}
\[\includegraphics[width=5.5in]{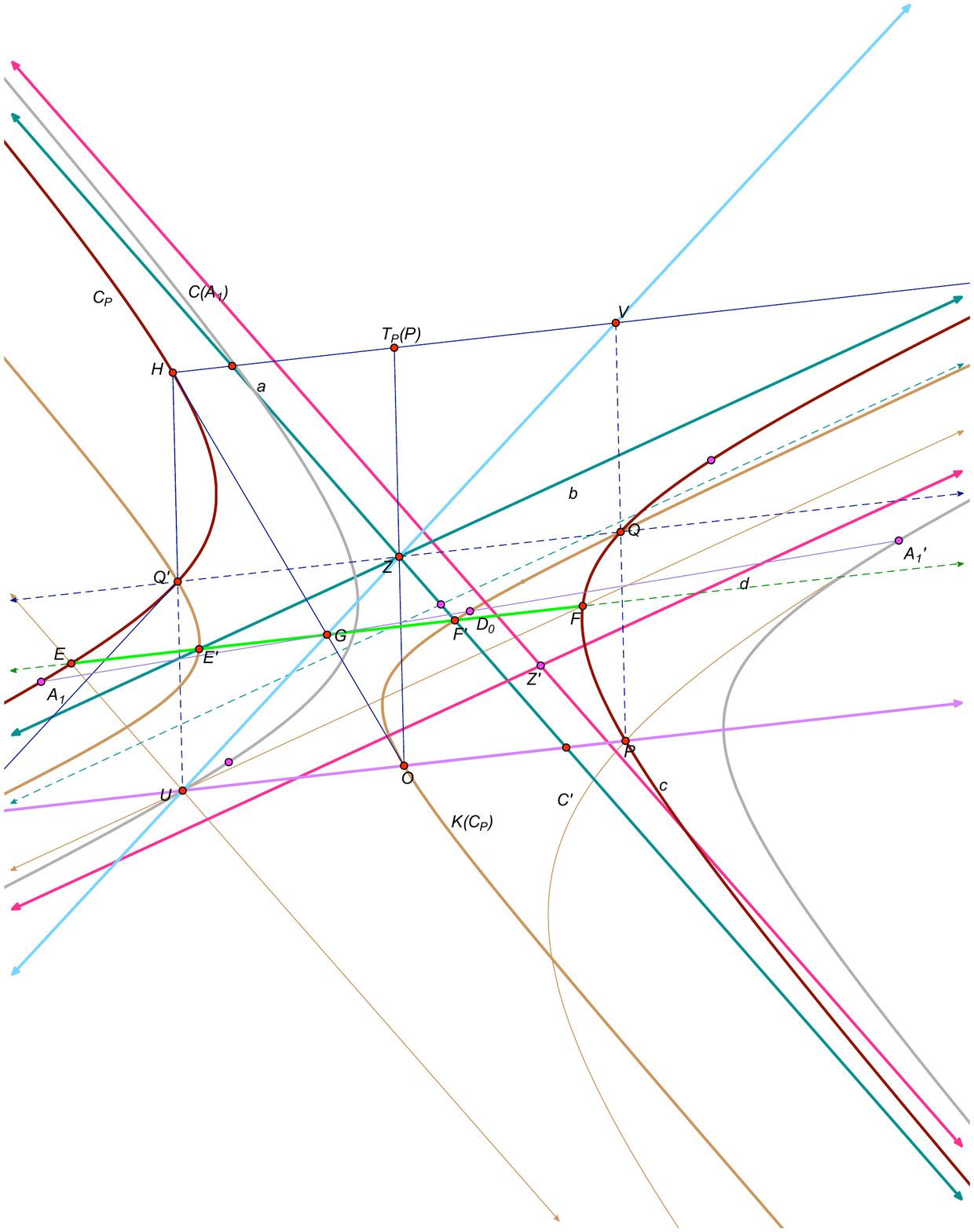}\]
\caption{Conics $\C$ (brown), $K(\cC_P)$ (tan), $\cC(A_1)$ (grey), $\cC'$ (light tan), with $A_1$ on the arc $\mathscr{E}$.}
\label{fig:4.2}
\end{figure}

\begin{figure}
\[\includegraphics[width=5.5in]{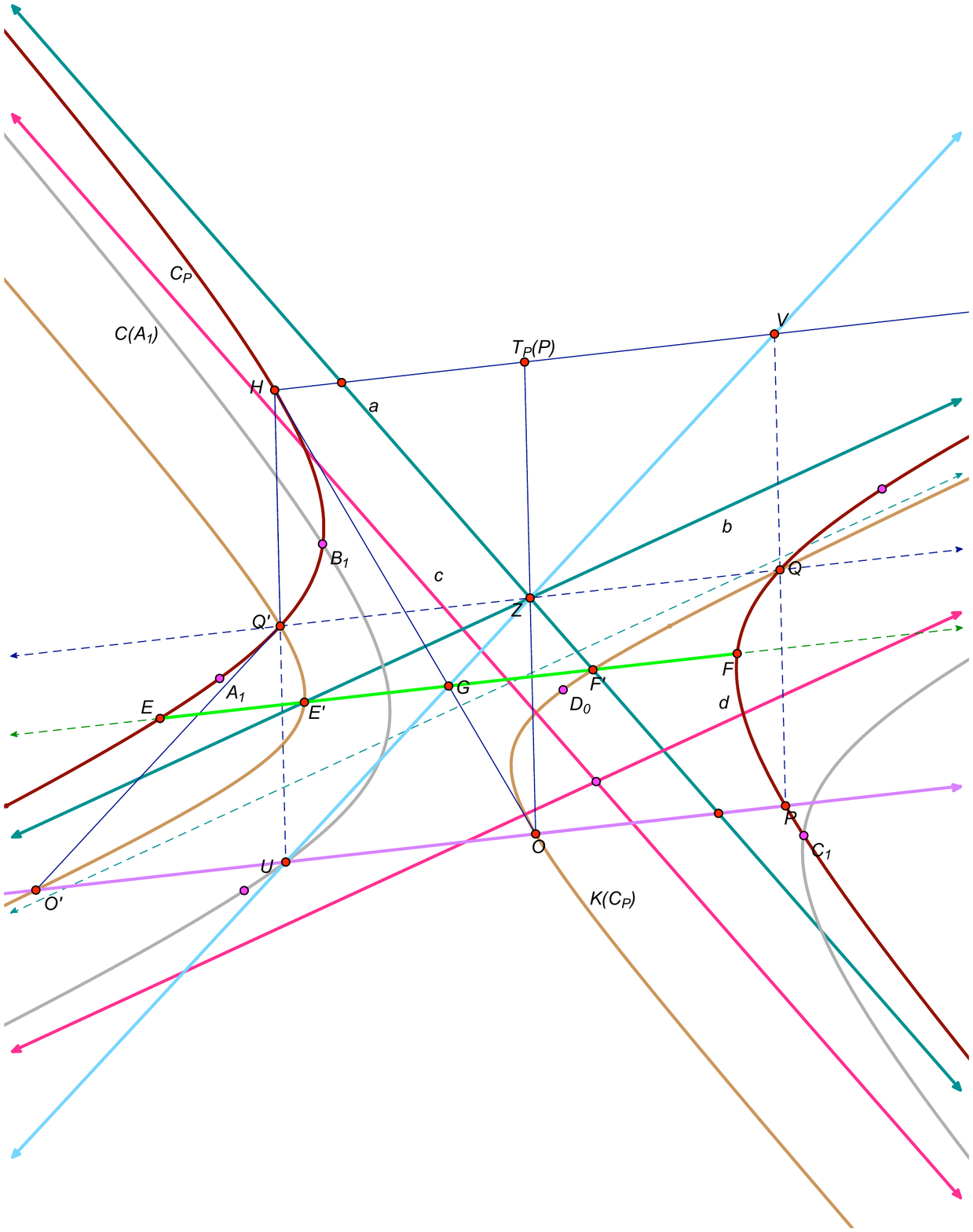}\]
\caption{Conics $\C$ (brown), $K(\cC_P)$ (tan), $\cC(A_1)$ (grey), with $A_1$ above $E$.}
\label{fig:4.3}
\end{figure}

We also claim that there are no ordinary points on $\C \cap \cC(A_1)$ when $A_1$ lies between $E$ and $P'$ on the open arc $\mathscr{E}=EP'$.  In this case $D_0=K(A_1)$ lies on the open arc of the conic $K(\cC)$ between $F'$ and $Q$.  For any $A_1$ on the left branch of $\C$, let $A_1'$ be the reflection of the point $A_1$ in $D_0$.  Using $D_0=K(A_1)$ it is easy to see that $A_1' = K^{-1}(A_1)$, so the tangent $\ell$ to $\C$ at $A_1$ is mapped to a parallel tangent $\ell'=K^{-1}(\ell)$ to $\cC'=K^{-1}(\cC)$ at $A_1'$.  On the other hand, $\ell$ is mapped by reflection in $D_0$ to a line through $A_1'$ parallel to $\ell$, so it must also be mapped to $\ell'$.  Therefore, $\ell'$ is tangent to both conics $\cC'$ and $\cC(A_1)=\textsf{R}_{D_0}(\cC)$ at $A_1'$.  Hence, these conics intersect with multiplicity $2$ at $A_1'$, and since their asymptotes are parallel, this is the only ordinary point where they can intersect.  This holds for any point $A_1$ on the left branch of $\C$, and therefore the right branch of the conic $\cC'$ is an {\it envelope} for the right branches of the conics $\cC(A_1)$.  For all $A_1$ on the left branch of $\C$, $D_0$ lies below the asymptote $K(b)$ of $K(\C)$, which is parallel to and lies halfway between $b$ and the asymptote $b'=K^{-1}(b)=UF$ of $\C'$; hence, the asymptote $d$ of $\C(A_1)$ lies below $b'$, implying that the right branch of $\C(A_1)$ lies in the interior of the right branch of $\C'$.  Since $\cC'$ intersects $\C$ at $K^{-1}(Q')=P$ and $K^{-1}(Q)=P'$, it crosses $\C$ at $P$, and points on $\cC'$ to the right of $P$ lie in the interior of $\C$.  For $A_1=E$, the right branches of $\C$ and $\cC(A_1)$ are also asymptotic in the direction of line $c=a$.  Therefore, as $A_1 \in \mathscr{E}$ moves from $E$ to $P'$, the right branch of $\cC(A_1)$ moves to the right, separated from the right branch of $\C$ by the asymptote $c$ and the conic $\cC'$ and remaining in the interior of $\C$.  It follows that the right branch of $\cC(A_1)$ contains no ordinary points of $\C$ for $A_1 \in \mathscr{E}$.  Since $D_0$ is in the exterior of $\C$ for these points $A_1$, any intersection of the left branch of $\cC(A_1)$ and the left branch of $\C$ would reflect through $D_0$ to an intersection of the right branches.  The left branch of $\cC(A_1)$ also does not intersect the right branch of $\C$ since it is the reflection of that branch in $D_0$.  This proves the claim.  \smallskip

Now assume $A_1$ lies outside the closed arc $\overline{\mathscr{E}}$ on the left branch of $\C$.  First, if $A_1$ lies above the point $E$, then $D_0$ lies below the point $F'$ on $K(\cC)$, and since $c$ lies on the same side of the line $a$ as $D_0$, $c$ lies to the left of the line $a$ in Figure \ref{fig:4.2}.  Since $c$ lies on $a \cdot l_\infty$ and is not the tangent $a$ at that point, it must intersect the conic $\C$ in a second, ordinary, point.  It cannot intersect the right branch of $\C$ because that branch is on the other side of the line $a$.  Hence, $c$ intersects the left branch of $\C$, whence it follows that the left branch of $\cC(A_1)$ intersects $\C$ as well (because this branch of $\cC(A_1)$ is asymptotic to an exterior ray of line $d$ in one direction and to $c$ in the other direction).  At the same time, this shows that the asymptote $a$ of $\C$ intersects the right branch of $\cC(A_1)$, so the right branch of $\C$ intersects the latter.  It is easy to see that these two intersection points are reflections of each other in the point $D_0$. \smallskip

On the other hand, if $A_1$ lies below the point $P'$ on the left branch of $\C$, then $D_0$ lies above $Q$ on $K(\cC)$.  Since $Q \in \cC \cap K(\cC)$, points to the right of $Q$ on $K(\cC)$ are in the interior of $\C$, so the reflection $Q_0$ of the point $Q$ in $D_0$ lies on the left branch of $\cC(A_1)$ in the interior of $\C$.  It follows that the left branch of $\cC(A_1)$ must intersect the right branch of $\C$ in two points.  The same arguments apply to points $A_1$ on the right branch of $\C$, and this completes the proof.
\end{proof}

\begin{lem}
The points $A_1=Q,Q'$ are the only points on $\C=PQHQ'P'$ for which $A_1$ lies on $\cC(A_1)$.
\label{lem:A1}
\end{lem}

\begin{proof} Certainly $Q' \in \cC(Q')$ because $D_0=K(Q')$ is the midpoint of segment $Q'P$, so $Q' = \textsf{R}_{D_0}(P)$ lies on $\cC(Q')$, the reflection of $\C$ in $D_0$.  The same argument holds for $Q$.  If $A_1$ is any point lying on $\cC(A_1)$, then $A_1$ and its reflection $A_1'$ in $D_0=K(A_1)$ both lie on the conic $\C$ and are collinear with the point $G$.  Since $A_1'=K^{-1}(A_1)$, the locus of points $A_1'$ coincides with the conic $\cC'=K^{-1}(\cC)$, whose asymptotes are parallel to the asymptotes of $\C$.  Hence, $\cC'$ intersects $\C$ in only the two points $P=K^{-1}(Q')$ and $P'=K^{-1}(Q)$.  This proves the lemma.
\end{proof}

We now fix a parallelogram $H_1U_1P_1V_1$ with center $Z_1$, distinguished point $G_1=U_1V_1 \cdot H_1O_1$, and its corresponding conic $\cC=P_1Q_1H_1Q_1'P_1'$, as in Figure \ref{fig:4.1}.  We will call this configuration the $P_1$ configuration, and consider it fixed for the following discussion.  \medskip

Let $ABC$ be a given triangle.  For any point $P$, not on a median of $ABC$, for which the map $\textsf{M}$, defined relative to $ABC$ and $P$, is a translation, there is an affine map $\textsf{A}^{-1}$ taking the parallelogram $HUPV$ for $ABC$ to the parallelogram $H_1U_1P_1V_1$.  (We avoid points on the medians of $ABC$, because for these points, the conic $\C_P=ABCPQ = AP \cup BC$ and parallelogram $HUPV$ are degenerate.)  Since $ABC$ is inscribed in the cevian conic $\cC_P=ABCPQ=PQHQ'P'$ for $P$, and the points $P', Q, Q'$ are defined by simple affine relationships in terms of the parallelogram $HUPV$, the image triangle $\textsf{A}^{-1}(ABC)=A_1B_1C_1$ under the map $\textsf{A}^{-1}$ is a triangle inscribed in the conic $\cC=P_1Q_1H_1Q_1'P_1'$.  By Theorem \ref{thm:trans} and Corollary \ref{cor:M} all the same relationships hold for the two configurations.  Hence, the centroid $G$ maps to the centroid $G_1$ in the $P_1$ configuration.  It follows from Lemma \ref{lem:arcs} and Lemma \ref{lem:A1} that the image $A_1$ of the point $A$ must lie in the complement of the union of closed arcs $\overline{\mathscr{E}}$ (from $E$ to $P_1'$) and $\overline{\mathscr{F}}$ (from $F$ to $P_1$) on $\cC$, and that $A_1$ is also distinct from the points $Q_1, Q_1'$ (as there is no triangle $A_1B_1C_1$ for these two points).  Thus,
\begin{equation}
A_1 \in \mathscr{A} = \cC - (\overline{\mathscr{E}} \cup \overline{\mathscr{F}} \cup \{Q_1,Q_1', A_\infty, B_\infty\}),
\label{eqn:arcs}
\end{equation}
where $A_\infty = a \cdot l_\infty$ and $B_\infty = b \cdot l_\infty$ are the infinite points on the asymptotes.  The set $\mathscr{A}$ is a union of $6$ open arcs on $\cC$. \medskip

Conversely, let $A_1 \in \mathscr{A}$ and let $A_1B_1C_1$ be the corresponding triangle inscribed in $\cC_{P_1}$.  Then the centroid of $A_1B_1C_1$ is $G_1$, and the cevian conic for $A_1B_1C_1$ and $P_1$ is $A_1B_1C_1P_1Q_1'=\cC_{P_1}$, coinciding with the conic $\cC=P_1Q_1H_1Q_1'P_1'$.  Moreover, the point $P_1$ does not lie on a median of $A_1B_1C_1$; otherwise one of the vertices of the triangle would be collinear with $P_1$ and $G_1$, implying that this vertex would have to coincide with $Q_1$ or $Q_1'$.  This conic has center $Z_1$, and the pole of $G_1Z_1$ is the point $V_\infty = P_1P_1' \cdot l_\infty$.  Now we use the characterization of the isotomic conjugate $\iota(P_1)$ (with respect to $A_1B_1C_1$) as the unique point different from $P_1$ lying in the intersection $\cC_{P_1} \cap P_1V_\infty$ = $\cC \cap P_1V_\infty$ to deduce that $\iota(P_1)=P_1'$.  (See II, p. 26.)  Theorem \ref{thm:trans} shows that $\textsf{M}_1=\textsf{M}_{P_1}$ for the triangle $A_1B_1C_1$ must be a translation.  If $\textsf{A}$ is an affine map taking $A_1B_1C_1$ to $ABC$, then Theorem \ref{thm:trans} shows once again that the map $\textsf{M}=\textsf{A}\textsf{M}_1\textsf{A}^{-1}$ is a translation for the point $P=\textsf{A}(P_1)$.  Hence, $P$ lies on the elliptic curve $\mathcal{E}_S$ of Section 3.  The argument of the previous paragraph shows that every point $P$ on $\mathcal{E}_S$, other than the $12$ points in its torsion group $T_{12}$, is the image $P=\textsf{A}(P_1)$ for some triangle $A_1B_1C_1$ inscribed in $\cC$ and an affine map $\textsf{A}$ for which $\textsf{A}(A_1B_1C_1)=ABC$.  This proves the following theorem.

\begin{thm} Fix a parallelogram $H_1U_1P_1V_1$ and the corresponding hyperbola $\cC=P_1Q_1H_1Q_1'P_1'$, as in Figure \ref{fig:4.1}.  The elliptic curve $\mathcal{E}_S$, minus the torsion subgroup $T_{12}$, corresponding to the vertices of $ABC$, the infinite points on its sides, and the points lying on the medians of $ABC$, is the locus of images $\textsf{A}(P_1)$, where $A_1$ is a point in the set $\mathscr{A} \subset \cC$ (a union of six open arcs on the hyperbola $\cC$), $B_1$ and $C_1$ are the unique points on $\cC$ for which triangle $A_1B_1C_1$ has centroid $G_1$, and $\textsf{A}$ is one of the two affine maps for which $\textsf{A}_1(A_1B_1C_1)=ABC$ or $\textsf{A}_2(A_1C_1B_1)=ABC$.
\end{thm}

By virtue of the above discussion, we have taken the situation of Figure \ref{fig:4.3}, where $P_1$ is fixed and the triangle $A_1B_1C_1$ varies, and transformed it, via the locus of maps $\textsf{A}$ corresponding to $A_1 \in \mathscr{A}$, to the fixed triangle $ABC$ and varying point $\textsf{A}(P_1)=P$ lying on the elliptic curve $\mathcal{E}_S$.

\end{section}

\noindent Dept. of Mathematics, Maloney Hall\\
Boston College\\
140 Commonwealth Ave., Chestnut Hill, Massachusetts, 02467-3806\\
{\it e-mail}: igor.minevich@bc.edu
\bigskip

\noindent Dept. of Mathematical Sciences\\
Indiana University - Purdue University at Indianapolis (IUPUI)\\
402 N. Blackford St., Indianapolis, Indiana, 46202\\
{\it e-mail}: pmorton@math.iupui.edu

\end{document}